\documentclass[English]{amsart}

\usepackage{amsmath} %
\usepackage{amsthm}
\usepackage{amsfonts}
\usepackage{amssymb}
\usepackage{mathtools}
\usepackage{stmaryrd}
\usepackage{mathrsfs}
\usepackage[bookmarks,hidelinks]{hyperref}
\usepackage[capitalize]{cleveref}
\usepackage{quiver}
\usepackage{adjustbox}

\newtheorem{thm}{Theorem}[section]
\newtheorem{prop}[thm]{Proposition}
\newtheorem{lem}[thm]{Lemma}
\newtheorem{cor}[thm]{Corollary}
\newtheorem{fact}[thm]{Fact}

\newtheorem{quest}[thm]{Question}
\theoremstyle{definition}
\newtheorem{defn}[thm]{Definition}

\theoremstyle{remark}

\newenvironment{claim}[1]{\par\noindent\emph{Claim.}\space#1}{}
\newenvironment{claimproof}[1]{\par\noindent\emph{Proof of claim.}\space#1}{\hfill $\vartriangleleft$}

\DeclareMathOperator{\varr}{var}
\DeclareMathOperator{\parr}{par}

\newcommand{\Rb}{\mathbb{R}}
\newcommand{\Fc}{\mathcal{F}}
\newcommand{\Uc}{\mathcal{U}}
\newcommand{\Rc}{\mathcal{R}}
\newcommand{\Xb}{\bar{X}}
\newcommand{\Ob}{\mathbb{M}}
\newcommand{\Lc}{\mathcal{L}}
\newcommand{\Gc}{\mathcal{G}}
\newcommand{\Vc}{\mathcal{V}}

\newcommand{\Pb}{\mathbb{P}}
\newcommand{\Ic}{\mathcal{I}}
\newcommand{\Sc}{\mathcal{S}}

\makeatletter \DeclareRobustCommand{\cset}{\@ifstar\star@cset\normal@cset}
\newcommand{\star@cset}[1]{{\left\llbracket#1\right\rrbracket}}
\newcommand{\normal@cset}[2][]{{\mathopen{#1\llbracket}#2\mathclose{#1\rrbracket}}}
\makeatother

\newcommand{\tleq}{\trianglelefteq}

\newcommand{\tgeq}{\trianglerighteq}
\newcommand{\tgess}{\triangleright}

\DeclareMathOperator{\res}{\upharpoonright}

\def\Ind{\setbox0=\hbox{$x$}\kern\wd0\hbox to 0pt{\hss$\mid$\hss}
  \lower.9\ht0\hbox to 0pt{\hss$\smile$\hss}\kern\wd0}

\def\Notind{\setbox0=\hbox{$x$}\kern\wd0\hbox to 0pt{\mathchardef
    \nn=12854\hss$\nn$\kern1.4\wd0\hss}\hbox to
  0pt{\hss$\mid$\hss}\lower.9\ht0 \hbox to 0pt{\hss$\smile$\hss}\kern\wd0}

\def\ind{\mathop{\mathpalette\Ind{}}}

\newcommand{\indf}{\ind^{\!\!\textnormal{f}}}

\newcommand{\indK}{\ind^{\!\!\textnormal{K}}}

\newcommand{\indi}{\ind^{\!\!\textnormal{i}}}

\newcommand{\inda}{\ind^{\!\!\textnormal{a}}}

\newcommand{\xbar}{\bar{x}}
\newcommand{\ybar}{\bar{y}}
\newcommand{\abar}{\bar{a}}
\newcommand{\dbar}{\bar{d}}
\newcommand{\bbar}{\bar{b}}
\newcommand{\cbar}{\bar{c}}

\DeclareMathOperator{\tp}{tp}

\newcommand{\e}{\varepsilon}

\DeclareMathOperator{\Aut}{Aut}
\DeclareMathOperator{\Autf}{Autf}
\DeclareMathOperator{\Th}{Th}

\newcommand{\Fraisse}{Fra\"\i ss\'e}

\newcommand{\toot}{\leftrightarrow}
\newcommand{\To}{\Rightarrow}

\newcommand{\upsett}[1]{\cset{#1}}%

\DeclareMathOperator{\acl}{acl}
\DeclareMathOperator{\dcl}{dcl}

\begin{document}

\title{Bi-invariant types, reliably invariant types, and the comb tree property}
\address{Department of Mathematics \\
Iowa State University \\
396 Carver Hall \\
411 Morrill Road \\
Ames, IA 50011, USA}
\author{James E. Hanson}
\email{jameseh@iastate.edu}
\date{\today}



\begin{abstract}
  \begin{sloppypar}
    We introduce and examine some special classes of invariant types---bi-invariant, strongly bi-invariant, extendibly invariant, and reliably invariant types---and show that they are related to certain model-theoretic tree properties.
  \end{sloppypar}
  
  We show that the comb tree property (recently introduced by Mutchnik) is equivalent to the failure of Kim's lemma for bi-invariant types and is implied by the failure of Kim's lemma for reliably invariant types over invariance bases. We show that every type over an invariance base extends to a reliably invariant type---generalizing a result of Kruckman and Ramsey---and use this to show that, under a reasonable definition of Kim-dividing, Kim-forking coincides with Kim-dividing over invariance bases in theories without the comb tree property. Assuming a measurable cardinal, we characterize the comb tree property in terms of a form of dual local character.

  We also show that the antichain tree property (introduced by Ahn and Kim) seems to have a somewhat similar relationship to strong bi-invariance. In particular, we show that NATP theories satisfy Kim's lemma for strongly bi-invariant types and (assuming a measurable cardinal) satisfy a different form of dual local character. Furthermore, we examine a mutual generalization of the local character properties satisfied by NTP$_2$ and NSOP$_1$ theories and show that it is satisfied by all NATP theories. 

  Finally, we give some related minor results---a strengthened local character characterization of NSOP$_1$ and a characterization of coheirs in terms of invariant extensions in expansions---as well as a pathological example of Kim-dividing.
\end{abstract}

\maketitle






\section*{Introduction}

In neostability theory, it is often useful when a combinatorial tameness property is found to be characterized by some form of Kim's lemma, which states that dividing along some indiscernible sequence in some class $A$ of indiscernible sequences entails dividing along all indiscernible sequences in some other class $B$. Some standard examples are the following:
\begin{itemize}
\item ($T$ simple) If $\varphi(x,b)$ divides over $M$, then $\varphi(x,b)$ divides along any non-forking Morley sequence in $\tp(b/M)$.
\item ($T$ NTP$_2$) If $\varphi(x,b)$ divides over $M$, then $\varphi(x,b)$ divides along any strict Morley sequence over $M$.
\item ($T$ NSOP$_1$) If $\varphi(x,b)$ Kim-divides over $M$, then $\varphi(x,b)$ divides along any sequence generated by an $M$-invariant type extending $\tp(b/M)$ and along any tree Morley sequence in $\tp(b/M)$.
\end{itemize}

An important additional consideration is the existence of the relevant special indiscernible sequences in the second class, possibly with extra restrictions, such as being indiscernible relative to some set of parameters. The modern proof of the symmetry of non-forking in simple theories uses the following fact:
  ($T$ simple) If $a \indf_M b$, then there is a Morley sequence $(a_i)_{i<\omega}$ which is $Mb$-indiscernible with $a_0 = a$. 
Likewise, in the context of NSOP$_1$ theories, symmetry of Kim-independence was shown by Kaplan and Ramsey using a similar but significantly harder to prove fact:
($T$ NSOP$_1$) If $a \indK_M b$, then there is a tree Morley sequence $(a_i)_{i<\omega}$ that is $Mb$-indiscernible with $a_0 = a$ \cite{KaplanRamseyOnKim}. 
For NTP$_2$ theories, on the other hand, while symmetry is not expected, similar machinery was used by Chernikov and Kaplan to show that forking and dividing coincide over extension bases. In particular, as part of this argument, they showed that in NTP$_2$ theories, any type over an invariance base extends to a \emph{strictly invariant type} \cite{Chernikov-Kaplan-NTP2}. Kruckman and Ramsey observed in  \cite{NKL} that Chernikov and Kaplan's proof almost doesn't rely on the assumption of NTP$_2$ and can be adapted to show that any type over a model extends to a \emph{Kim-strictly invariant type}.

\begin{defn}
  An $M$-invariant type $p(x)$ is \emph{strictly invariant} if whenever $a \models p \res M b$, $b \indf_M a$. $p(x)$ is \emph{Kim-strictly invariant} if whenever $a \models p \res M b$, $b \indK_M a$.
\end{defn}

\begin{fact}[{Kruckman, Ramsey \cite[Thm. 2.26]{NKL}}]\label{fact:Kruckman-Ramsey}
  ($T$ arbitrary) Any type over a model $M$ has a Kim-strictly invariant extension.
\end{fact}

In \cref{cor:reliability-is-reliable}, we generalize \cref{fact:Kruckman-Ramsey} by showing that any type over an invariance base extends to a Kim-strictly invariant type.

\cref{fact:Kruckman-Ramsey} prompted Kruckman and Ramsey to investigate the following variant of Kim's lemma as possibly characterizing a good mutual generalization of NTP$_2$ and NSOP$_1$: A theory $T$ satisfies \emph{new Kim's lemma} if whenever $\varphi(x,b)$ Kim-divides over a model $M$, then it Kim-divides with regards to any Kim-strictly invariant type extending $\tp(b/M)$. In \cite{NKL}, Kruckman and Ramsey gave some natural examples of theories with both TP$_2$ and SOP$_1$ which nevertheless satisfy new Kim's lemma, specifically the generic theory of parameterized dense linear orders and the two-sorted theory of an infinite-dimensional vector space over a real-cloesd field with a generic (alternating or symmetric) bilinear form. As shown in \cite{NKL}, a failure of new Kim's lemma entails the existence of a certain combinatorial configuration mutually generalizing TP$_2$ and SOP$_1$, which they call the \emph{bizarre tree property} or \emph{BTP}, but whether the converse holds is unclear at the moment. 

The only other previously known general construction of (Kim-)strictly invariant types seems to have been the following fact: 

\begin{defn}
  A global type $p(x)$ is an \emph{$M$-coheir} or a \emph{coheir over $M$} if it is finitely satisfiable in $M$. $p(x)$ is an \emph{$M$-heir} or an \emph{heir over $M$} if for every $M$\nobreakdash-\hspace{0pt}formula $\varphi(x,y)$, if there is a $b$ in the monster such that $\varphi(x,b) \in p(x)$, then there is a $c \in M$ such that $\varphi(x,c) \in M$.
\end{defn}

\begin{fact}\label{fact:heir-coheir-basic}
  If $p(x)$ is $M$-invariant and $N \succeq M$ is $(|M|+|\Lc|)^+$-saturated, then $p^{\otimes n}$ is an $N$-heir for every $n<\omega$.
\end{fact}

In particular, if $p(x)$ is an $M$-heir and $a \models p \res M b$, then $\tp(b/Ma)$ extends to an $M$-coheir, implying that any type that is both invariant and an heir is strictly invariant. But the property in \cref{fact:heir-coheir-basic} is ostensibly stronger than mere strict invariance. 

\begin{defn}
  A global type $p(x)$ is an \emph{$M$-heir-coheir} if $p(x)$ is an $M$-heir and an $M$-coheir. $p(x)$ is \emph{$M$-bi-invariant} if it is $M$-invariant and whenever $a \models p \res M b$, then $\tp(b/Ma)$ extends to a global $M$-invariant type. $p(x)$ is \emph{strongly $M$-bi-invariant} if $p^{\otimes n}$ is $M$-bi-invariant for every $n<\omega$.
\end{defn}

Note that in an NIP theory, any strictly invariant type is bi-invariant.

One of the contributions of this paper will be to present a couple of novel methods for constructing heir-coheirs, and therefore bi-invariant types. Unlike \cref{fact:heir-coheir-basic}, the heir-coheirs we construct will not be strongly bi-invariant. In fact, there seems to be a significant difference between the tasks of constructing bi-invariant and strongly bi-invariant types. In particular, bi-invariance is something that can be accomplished generically on the level of formulas. This can be seen in the following proposition (which we will not use elsewhere in this paper but is motivating and may be of independent interest).

\begin{prop}\label{prop:easy-construction}
  Let $T$ be a countable theory. Let $M\models T$ be a countable model with the following weak saturation property:
  \begin{itemize}
  \item[$ $] For any $M$-formula $\varphi(x,y)$, if there is an element $b$ in the monster such that $\varphi(M,b)$ is infinite, then there is a $c \in M$ such that $\varphi(M,c)$ is infinite.
  \end{itemize}
  Let $S_1^{\mathrm{nr}}(M)$ be the set of non-realized $1$-types over $M$. There is a dense $G_\delta$ set $X \subseteq S_1^{\mathrm{nr}}(M)$ such that for any $p \in X$, any $M$-coheir extending $p$ is an $M$-heir.
\end{prop}
\begin{proof}
  For any $M$-formulas $\varphi(x)$ and $\psi(x,y)$, if there is a $b$ in the monster such that $\varphi(M)\wedge \psi(M,b)$ is infinite, find a $c \in M$ such that $\varphi(M) \wedge \psi(M,c)$ is infinite and let $U_{\varphi,\psi}$ be the set of types in $S^{\mathrm{nr}}_1(M)$ containing $\varphi(x) \wedge \psi(x,c)$. (Note that this set is non-empty since $\varphi(M) \wedge \psi(M,c)$ is infinite.) If there is no such $c$, let $U_{\varphi,\psi}$ be the set of types in $S^{\mathrm{nr}}_1(M)$ containing $\varphi(x)$. Let $U_{\psi}= \bigcup_{\varphi \in \Lc(M)}U_{\varphi,\psi}$. $U_{\psi}$ is a dense open subset of $S^{\mathrm{nr}}_1(M)$ for each $\psi(x,y)$. Let $X = \bigcap_{\psi \in \Lc(M)}U_\psi$. $X$ is a dense $G_\delta$ set.

  Now fix $p(x) \in X$. We need to argue that any $M$-coheir $q(x) \supset p(x)$ is an $M$\nobreakdash-\hspace{0pt}heir. Suppose that $\psi(x,b) \in q(x)$. Since $q(x)$ is an $M$-coheir that is not realized in $M$, $\varphi(M) \wedge \psi(M,b)$ must be infinite for every $M$-formula $\varphi(x) \in p(x)$. Since $p(x) \in U_\psi$, we have that for some $c \in M$, $\psi(x,c) \in p(x)$. Since we can do this for any $M$-formula $\psi(x,y)$, we have that $q(x)$ is an $M$-heir.
\end{proof}

Note that the saturation property in \cref{prop:easy-construction} is satisfied by any computably saturated model and by any model of a theory that eliminates $\exists^\infty$. It's also relatively easy to see that \cref{prop:easy-construction} can fail in models at least as large as the covering number of the ideal of meager sets. Specifically, no coheir of a type over $(\Rb,<)$ is also an heir.

In \cite{Mutchnik-NSOP2}, Mutchnik introduced a certain combinatorial configuration he called $\omega$\nobreakdash-\hspace{0pt}DCTP$_2$ as part of his proof that NSOP$_1$ and NSOP$_2$ (or equivalently NTP$_1$) are the same. For the sake of simplicity, we will just refer to this condition as the \emph{comb tree property} or \emph{CTP} (\cref{defn:CTP}). CTP will be a major focus of this paper. 

In \cite{ATP-1}, Ahn and Kim introduced the \emph{antichain tree property} or \emph{ATP}, which, like BTP, is intended to be a good mutual generalization of TP$_2$ and SOP$_1$. In \cite{Ahn2022}, they show that ATP is always witnessed by a formula with a single free variable and is always witnessed by the $2$-inconsistent version of the condition. Furthermore, in \cite{Some-Remarks-Kim-dividing-NATP}, Kim and Lee show that in NATP theories, Kim-forking and Kim-dividing coincide over models and that if new Kim's lemma holds (in the sense of Kruckman and Ramsey), then coheirs are universal witnesses of Kim-dividing if and only if they are Kim-strict.

It is not hard to show\footnote{The fact that ATP implies CTP is immediate from the definition, as discussed in \cite{Ahn2022}, and the fact that CTP implies BTP is a corollary of unpublished results of Kruckman and Ramsey, but also follows from out \cref{cor:NKL-to-NATP}.


} that these three conditions are related. In particular (and in an alphabetically frustrating way), ATP implies CTP, which implies BTP.

The main contribution of this paper will be a characterization of NCTP in terms of a variant of new Kim's lemma (where Kim-strict invariance is replaced with bi-invariance). We will also give an argument that NATP implies the analogous property with strong bi-invariance, although the converse is unclear.

In \cref{sec:reliably}, we will argue that a certain definition of Kim-dividing is natural over invariance bases and extend a result of Mutchnik's \cite[Thm.~4.9]{Mutchnik-NSOP2} and a result of Kim and Lee's \cite[Rem.~5.10]{Some-Remarks-Kim-dividing-NATP} by showing that under this definition in NCTP theories, Kim-forking coincides with Kim-dividing over invariance bases. We do this by introducing the notions of \emph{reliably invariant types} and \emph{reliable coheirs} and show that these always exist over invariance bases and models, respectively, and that these always witness Kim-dividing in CTP theories. This partially remedies a deficiency of our characterization which is that not all types over models extend to heir-coheirs or even strictly invariant types (see \cite[Sec.~5.1]{Chernikov-Kaplan-NTP2}). A remaining issue with this is that it is unclear whether NCTP is actually characterized by Kim's lemma for reliable coheirs.

In \cref{sec:local-char}, we give a dual local character characterization of NCTP assuming the existence of a measurable cardinal. We show that NATP theories satisfy a similar form of dual local character, but again the converse is unclear. In \cref{sec:natp-implies-generic}, we also discuss a notion of local character mutually generalizing the local character properties that characterize NTP$_2$ and NSOP$_1$ theories and show that it is implied by NATP.

We summarize the known implications between mutual generalizations of NTP$_2$ and NSOP$_1$ in Figure~\ref{fig:knowledge}.

\begin{figure}
  \centering
  \adjustbox{scale=0.95}{
      \begin{tikzcd}
	& {\text{NBTP}} \\
	& {\text{New Kim's Lemma}} \\
	\begin{array}{c} \text{Bi-invariant} \\ \text{Kim's Lemma} \end{array} & {\text{NCTP}} & \begin{array}{c} \text{Bi-invariant} \\ \text{Dual Local Character} \end{array} \\
	\begin{array}{c} \text{Reliably Invariant} \\ \text{Kim's Lemma} \end{array} & {\text{NATP}} &  \\
	 & \begin{array}{c} \text{Strongly Bi-invariant} \\ \text{Kim's Lemma} \end{array} &  \\
    \begin{array}{c} \text{Generic Stationary} \\ \text{Local Character} \end{array} &  & \begin{array}{c} \text{Strongly Bi-invariant} \\ \text{Dual Local Character} \end{array}
	\arrow["{\text{\cite[Thm.~5.2]{NKL}}}", from=1-2, to=2-2]
	\arrow["{\text{Cor.~\ref{cor:NKL-to-NATP}}}"', from=2-2, to=3-2]
	\arrow["{\text{Thm.~\ref{thm:CTP-char}}}"', tail reversed, from=3-2, to=3-1]
	\arrow["{\text{Prop.~\ref{prop:dual-char}}}", shift left=3, dashed, from=3-2, to=3-3]
	\arrow["{\text{Prop.~\ref{prop:reliable-main}}}", from=3-2, to=4-1]
	\arrow[from=3-2, to=4-2]
	\arrow["{\text{Prop.~\ref{prop:char-uncountable-languages}}}", shift left=3, from=3-3, to=3-2]
	\arrow["{\text{Prop.~\ref{prop:Kim-failure-implies-CTP}}}", from=4-2, to=5-2]
	\arrow["{\text{Prop.~\ref{prop:gen-stat-ATP}}}"', from=5-2, to=6-1]
	\arrow["{\text{Prop.~\ref{prop:dual-char}}}", dashed, from=5-2, to=6-3]
\end{tikzcd}
}  
  \caption{Known implications between mutual generalizations of NTP$_2$ and NSOP$_1$. (The dashed lines are proven assuming the existence of a measurable cardinal.)}
  \label{fig:knowledge}
\end{figure}

Finally, we would like to thank Alex Kruckman and Nicholas Ramsey for many valuable discussions regarding the ideas in this paper.


\section{The comb tree property}

Given $\sigma \in 2^{<\omega}$, we write $\upsett{\sigma}$ for the set of $\tau \in 2^{<\omega}$ such that $\tau \tgeq \sigma$.

\begin{defn}\label{defn:CTP}
  Given an ordinal $\alpha$, a set $X \subseteq 2^{<\alpha}$ is a \emph{right-comb} 
  if it is an antichain and satisfies that for any $\sigma \in 2^{<\alpha}$, if there a $\tau \in X$ with $\tau \tgeq \sigma \frown 1$, then there is at most one $\tau \in X$ with $\tau \tgeq \sigma \frown 0$.

  A theory $T$ has \emph{$k$-CTP} if there is a binary tree $\{b_{\sigma}\}_{\sigma \in 2^{<\omega}}$ and a formula $\varphi(x,y)$ such that for any path $\alpha \in 2^\omega$,  $\{\varphi(x,b_{\alpha \res n}) : n < \omega\}$ is $k$-inconsistent but for any right-comb $X \subseteq 2^{<\omega}$, $\{\varphi(x,b_{\sigma}) : \sigma \in X\}$ is consistent.

  $T$ has \emph{CTP} if it has $k$-CTP for some $k<\omega$.
\end{defn}

Note that since any right-comb is an antichain, any theory with ATP has CTP, as observed in \cite[Rem.~5.7]{Some-Remarks-Kim-dividing-NATP}. Another thing to note is that the dual condition of CTP (where paths are consistent and right-combs are $k$-inconsistent), was shown to be equivalent to NSOP$_1$ in \cite{Mutchnik-NSOP2}. This means that many of our results dualize to give analogous results for NSOP$_1$ theories. We collect these in \cref{sec:Dual-NSOP1}.

\subsection{Failure of Kim's lemma for heir-coheirs from the comb tree property}

\begin{defn}
  A subset $X \subseteq 2^{<\omega}$ is \emph{dense above $\sigma$} if for every $\tau \tgeq \sigma$, $X\cap \upsett{\tau}$ is non-empty. $X$ is \emph{somewhere dense} if there is a $\sigma$ such that $X$ is dense above $\sigma$. A filter $\Fc$ on $2^{<\omega}$ is \emph{everywhere somewhere dense} if every $X \in \Fc$ is somewhere dense.
\end{defn}

In this paper, we will only ever use the term `somewhere dense' to refer to the above property of a subset of a tree. We will never use it in the topological sense. We will also only use the term `filter' to refer to proper filters (i.e., filters that do not contain $\varnothing$).

The following is a fairly standard idea in forcing.

\begin{lem}\label{lem:dense-pigeons}
  For any somewhere dense $X \subseteq 2^{<\omega}$ and any $Y \subseteq 2^{<\omega}$, either $X \cap Y$ is somewhere dense or $X \setminus Y$ is somewhere dense.
\end{lem}
\begin{proof}
  Fix $\sigma$ such that $X$ is dense over $\sigma$. Suppose that for every $\tau \tgeq \sigma$, $X \cap Y$ fails to be dense above $\tau$. Then for every $\tau \tgeq \sigma$, $(X\setminus Y)\cap \upsett{\tau}$ is non-empty, so $X\setminus Y$ is dense above $\sigma$.
\end{proof}

\begin{lem}\label{lem:everything-everywhere-all-at-once}
  Every everywhere somewhere dense filter extends to some everywhere somewhere dense ultrafilter.
\end{lem}
\begin{proof}
  By transfinite induction it is sufficient to show that if $\Fc$ is an everywhere somewhere dense filter and $Y \subseteq 2^{<\omega}$, then either $\Fc\cup \{Y\}$ or $\Fc \cup \{2^{<\omega}\setminus Y\}$ generates an everywhere somewhere dense filter. So fix some such $Y$. By \cref{lem:dense-pigeons} we have that for each $X \in \Fc$, either $X\cap Y$ is somewhere dense or $X \setminus Y$ is somewhere dense. One of the sets  $\{X \in \Fc: X\cap Y~\text{somewhere dense}\}$ and $\{X \in \Fc: X \setminus Y~\text{somewhere dense}\}$ must be cofinal in $\Fc$. Assume without loss of generality that $\{X \in \Fc: X\cap Y~\text{somewhere dense}\}$ is cofinal in $\Fc$. This implies that actually $X\cap Y$ is somewhere dense for all $X \in \Fc$. Therefore $\Fc \cup \{Y\}$ generates an everywhere somewhere dense filter.
\end{proof}

We will see later in \cref{prop:char-uncountable-languages} that the following \cref{prop:countable-characterization-CTP} holds even for uncountable languages. That said, the proof in uncountable languages is a bit more technical, so we feel it is appropriate to present the friendlier proof for countable theories first.

\begin{prop}\label{prop:countable-characterization-CTP}
  Let $T$ be a countable theory. If $T$ has CTP, then there is a countable model $M$, a formula $\varphi(x,y)$, an $M$-heir-coheir $p(y)$, and an $M$-coheir $q(y)$ such that $p \res M = q \res M$ and $\varphi(x,y)$ $q$-Kim-divides but does not $p$-Kim-divide.
\end{prop}
\begin{proof}
  Let $\varphi(x,y)$ and $(b_\sigma)_{\sigma \in 2^{<\omega}}$ witness that $T$ has CTP.

  Let $(k(n),m(n))_{n<\omega}$ be an enumeration of $\omega^2$ with the property that for each pair $\ell,o<\omega$, the set $\{n<\omega:\langle k(n),m(n) \rangle = \langle \ell,o \rangle\}$ is infinite.

  Fix a countable model $M_0 \supseteq \{b_\sigma : \sigma \in 2^{<\omega}\}$. Let $\sigma_0 = \varnothing$ and let $X_0 = \upsett{1}$. Note that $X_0$ is dense over $\sigma_0\frown 1$ and $X_0 \subseteq \upsett{\sigma_0\frown 1}$. (This will be our induction hypothesis.)

  At stage $n < \omega$, suppose we have a countable model $M_n$ and $\sigma_n$ and $X_n$ such that $X_n$ is dense over $\sigma_n\frown 1$ and $X_n \subseteq \upsett{\sigma_n\frown 1}$. Let $(\psi_{k(\ell)}^{n}(y,z))_{\ell<\omega}$ be an enumeration of all $M_n$-formulas. To get $M_{n+1}$, $\sigma_{n+1}$, and $X_{n+1}$, perform the following construction:
  \begin{itemize}
  \item If $\psi^{m(n)}_{k(n)}(y,z)$ has already been defined and there is a $c$ in the monster such that $X_n\cap\{\sigma:\psi^{m(n)}_{k(n)}(b_\sigma,c)\}$ is somewhere dense: Fix some such $c$. Let $M_{n+1}\supseteq M_nc$ be a countable model. Find $\tau\tgeq \sigma_n\frown 1$ such that $X_{n+\frac{1}{2}}\coloneqq X_n\cap\{\sigma:\psi^{m(n)}_{k(n)}(b_\sigma,c)\}$ is dense over $\tau$. Find $\sigma_{n+1} \tgeq \tau$ such that $b_{\sigma_{n+1}} \in X_{n+\frac{1}{2}}$, and let $X_{n+1} = X_{n+\frac{1}{2}}\cap\upsett{\sigma_{n+1}\frown 1}$.
  \item If the condition in the previous bullet point fails: Let $M_{n+1} = M_n$. Find $\sigma_{n+1}\tgeq \sigma_n\frown 1$ such that $b_{\sigma_{n+1}} \in X_n$. Let $X_{n+1} = X_n \cap \upsett{\sigma_{n+1}\frown 1}$.
  \end{itemize}
  Note that in both cases we have ensured that $b_{\sigma_{n+1}} \in X_n$.

  After the construction is completed, let $M = \bigcup_{n<\omega}M_n$. Let $\Fc$ be the filter generated by $\{X_n:n<\omega\}\cup\left\{ Y \right\}$ where $Y \coloneqq \bigcup_{n<\omega}\upsett{\sigma_n\frown 0}$. Note that since each $X_n$ is dense above $\sigma_{n+1}$, this filter is everywhere somewhere dense. Let $\Uc$ be an everywhere somewhere dense ultrafilter extending $\Fc$. Let $p(y)$ be the global $M$-coheir corresponding to $\Uc$ (i.e., $p(y) = \{\psi(y,c):\{\sigma \in 2^{<\omega} : \psi(b_\sigma,c)\} \in \Uc\}$). Let $q(y)$ be any non-realized global $M$-coheir finitely satisfiable in $\{b_{\sigma_n}:n<\omega\}$. Note that $\{\varphi(x,b_{\sigma_n}) : n < \omega\}$ is uniformly inconsistent (since the $\sigma_n$'s form a path in $2^{<\omega}$). Therefore $\varphi(x,y)$ must $q$-Kim-divide.

  \vspace{1em}

  \begin{claim}
    $p\res M = q \res M$.
  \end{claim}
  \begin{claimproof}
    Fix an $M$-formula $\psi(y)\in p(y)$. $\psi(y)$ is actually an $M_k$-formula for some $k<\omega$. By the choice of our enumeration, this means that there was a stage $n$ at which $\psi^{m(n)}_{k(n)}(y,z)$ was defined and equal to $\psi(y)$ (where $z$ is a dummy variable). Since $\psi(y) \in p(y)$ and since $\Uc$ is everywhere somewhere dense, we must have chosen the first bullet point at this stage. Hence $\psi(b_\sigma)$ holds for all $\sigma \in X_{n+1}$, whereby $\psi(b_{\sigma_\ell})$ holds for all $\ell > n+1$. Therefore $\psi(y) \in q(y)$ as well. Since we can do this for every $M$-formula $\psi(y) \in p(y)$, we have that $p \res M = q \res M$.
  \end{claimproof}

  \vspace{1em}

  \begin{claim}
    $p$ is an heir over $M$.
  \end{claim}
  \begin{claimproof}
    Fix an $M$-formula $\psi(y,z)$. Once again, there must have been a stage $n$ at which $\psi^{m(n)}_{k(n)}$ was defined and equal to $\psi$. Suppose that there is a $d$ in the monster such that $\psi(y,d) \in p(y)$. Since $\Uc$ is everywhere somewhere dense, this implies that we chose the first bullet point at stage $n$, so we found some $c$, added this $c$ to $M_{n+1}$, and moved to a set $X_{n+1}$ satisfying $\psi(b_{\sigma},c)$ for all $\sigma \in X_{n+1}$. Therefore $\psi(y,c) \in p(y)$. Since we can do this for any $M$-formula $\psi(y,z)$, we have that $p$ is an heir over $M$.
  \end{claimproof}

  \vspace{1em}

  \begin{claim}
    $\varphi(x,y)$ does not $p$-Kim-divide.
  \end{claim}
  \begin{claimproof}
    Recall that $Y \coloneqq \bigcup_{n<\omega}\upsett{\sigma_n\frown 0}$ is in $\Uc$. Also note that by construction, $\upsett{\sigma_n} \in \Uc$ for each $n<\omega$. Let $Y_n = Y \cap \upsett{\sigma_n}$ for each $n<\omega$. Note that if $n_0,n_1,\dots,n_{k-1}$ is an increasing sequence of integers and if $\tau_i \tgeq \sigma_{n_i}\frown 0$ for each $i<k$, then $\{\tau_i : i < k\}$ is a right-comb.

    Let $(e_k)_{k<\omega}$ be a Morley sequence generated by $p$. We have by assumption that for any finite right-comb $Z \subseteq Y$, $\{\varphi(x,b_\tau):\tau \in Z\}$ is consistent. Suppose that for some $k<\omega$ (possibly $0$), we've shown that for every finite right-comb $Z \subseteq Y$, $\{\varphi(x,b_\tau):\tau \in Z\}\cup\{\varphi(x,e_i):i<k\}$ is consistent. For any such $Z$, there is an $n$ such that for any $\eta \in Y_n$, $Z \cup \{\eta\}$ is a right-comb. (In particular, this will be true for any sufficiently large $n$.) This implies that $\{\varphi(x,b_\tau):\tau \in Z\} \cup \{\varphi(x,e_i): i < k+1\}$ is consistent. Hence, by induction, we have that $\{\varphi(x,e_i): i < \omega\}$ is consistent and so $\varphi(x,y)$ does not $p$-Kim-divide.
  \end{claimproof}

\vspace{1em}
  
  Therefore $p$ and $q$ satisfy the required conditions.
\end{proof}

One thing to note is that instead of building the model in the proof of \cref{prop:countable-characterization-CTP} at the same time as the filter, we could instead have built a fixed model satisfying the following weak saturation property that is analogous to the one found in \cref{prop:easy-construction}:
\begin{itemize}
\item[$ $] For every $M$-formula $\varphi(x,y)$ and $\sigma \in 2^{<\omega}$, if there is a $c$ in the monster such that $\{\tau \in 2^{<\omega} : \varphi(b_\tau,c)\}$ is dense above $\sigma$, then there is a $d\in M$ such that $\{\tau \in 2^{<\omega} : \varphi(b_\tau,d)\}$ is dense above $\sigma$.
\end{itemize}
This ends up being a bit more work to state than the given proof of \cref{prop:countable-characterization-CTP}, but this perspective highlights the similarity between \cref{prop:countable-characterization-CTP} and \cref{prop:easy-construction}.

Something that is frustrating and interesting is that, at least to the present author, there doesn't seem to be a clear way to prove the ATP analog of \cref{prop:countable-characterization-CTP}. That is to say, it is unclear if one can use an instance of ATP to build a failure of Kim's lemma for strongly bi-invariant types.

\begin{quest}\label{quest:ATP-Kim-char}
  If $T$ has ATP, does it follow that there is a model $M$, a formula $\varphi(x,y)$, a strongly $M$-bi-invariant type $p(y)$, and an $M$-invariant type $q(y)$ such that $p \res M = q \res M$ and $\varphi(x,y)$ $q$-Kim-divides but does not $p$-Kim-divide?
\end{quest}


\subsection{Characterization of the comb tree property}

\begin{prop}\label{prop:Kim-failure-implies-CTP}
  Suppose there is a formula $\varphi(x,y)$ and two $A$-invariant types $p(y)$ and $q(y)$ such that $p \res A= q\res A$ and $\varphi(x,y)$ $q$-Kim-divides but does not $p$-Kim-divide.
  \begin{enumerate}
  \item\label{can-has-CTP} If $p$ is bi-invariant, then $T$ has CTP.
  \item\label{can-has-ATP} If $p$ is strongly bi-invariant, then $T$ has ATP.
  \end{enumerate}
\end{prop}
\begin{proof}
For \ref{can-has-CTP}, suppose that $\{\varphi(x,e_i):i < \omega\}$ is $k$-inconsistent for any Morley sequence $e_{<\omega}$ generated by $q$.
  
  We will argue by induction that for any $n<\omega$, there is a family $(b_\sigma)_{\sigma \in 2^{\leq n}}$ of parameters realizing $p\res A$ such that for each $\sigma \in 2^{<n}$,
  \begin{itemize}
  \item for each $\tau \tgeq \sigma \frown 0$, $b_\tau \models p \res A \cup \{b_{\eta}:\eta \tgeq \sigma\frown 1\}$ and
  \item $b_\sigma \models q \res A \cup \{b_\eta : \eta \tgess \sigma\}$.
  \end{itemize}
  Note that the second condition clearly implies that any path through such a tree is a (reverse) Morley sequence generated by $q$ (so in particular, $\{\varphi(x,b_{\eta \res i}) : i \leq n\}$ is $k$-inconsistent for any $\eta \in 2^n$). The first condition, moreover, implies that any right-comb in the tree is a Morley sequence generated by $p$ (in some enumeration). Therefore for any right-comb $X\subseteq 2^{\leq n}$, $\{\varphi(x,b_\sigma):\sigma \in X\}$ is consistent. So if we can show that these trees exist for all $n$, we will have established that $T$ has $k$-CTP.

  For $n = 1$, the condition is trivial, since there is only a single $\sigma$ in the tree.

  Suppose we have built such a tree $(c_\sigma)_{\sigma \in 2^{\leq n}}$ for some $n$. Start the next tree by setting $b_{0\frown \sigma} = c_\sigma$ for each $\sigma \in 2^{\leq n}$. Find $d \models p \res A \cup \{b_\sigma : \sigma \in 2^{\leq n+1},~\sigma\tgeq 0\}$. Since $p$ is $A$-bi-invariant, we can find an $A$-invariant type $r(\xbar)$ extending $\tp(b_{\tgeq 0}/Ad)$. Now find $b_{\tgeq 1}$ such that $b_{\tgeq 1} \models r \res Ab_{\tgeq 0}$. By $A$-invariance of $r$, we have that for each $\sigma \tgeq 0$, $b_\sigma \models p \res A b_{\tgeq 1}$. Finally, let $b_{\varnothing} \models q \res A b_{\tgess \varnothing}$.

  Since we can do this for any $n$, by induction, we have that $T$ has $k$-CTP.

  For \ref{can-has-ATP}, then we can build similar trees $(b_\sigma)_{\sigma \in 2^{\leq n}}$ satisfying the additional property that
  \begin{itemize}
  \item
    for each antichain $B$ of elements of $2^{\leq n}$, $\{b_\sigma : \sigma \in B\}$ is a Morley sequence in $p$ over $A \cup \{b_\eta  : \eta \tgeq \sigma \frown 1\}$.
  \end{itemize}
  To see that this is possible, we just need to modify the induction step. Suppose we have built such a tree $(c_\sigma)_{\sigma \in 2^{\leq n}}$ for some $n$. Start the next tree by setting $b_{0 \frown \sigma} = c_\sigma$ for each $\sigma \in 2^{\leq n}$. Find $\dbar \models p^{\otimes \omega}\res A \cup \{b_\sigma : \sigma \in 2^{\leq n+1},~\sigma \tgeq 0\}$. Since $p^{\otimes \omega}$ is $A$-bi-invariant, we can find an $A$-invariant type $r(\xbar)$ extending $\tp(b_{\tgeq 0}/A \dbar)$. Then we can find $b_{\tgeq 1}$ such that $b_{\tgeq 1} \models r \res Ab_{\tgeq 0}$. Now, let $B$ be an antichain of elements of $2^{\leq n+1}$. If $B = \{\varnothing\}$, then the statement is trivial. Otherwise, we have that $B\cap \cset{0}$ and $B \cap \cset{1}$ are each antichains. By the induction hypothesis, this means that $\bbar^0 = \{b_\sigma : \sigma \in B\cap \cset{0}\}$ is a Morley sequence in $p$ and $\bbar^1 = \{b_\sigma : \sigma \in B\cap \cset{1}\}$ is a Morley sequence in $p$. By construction, this means that $\bbar^0$ realizes the same type over $A$ as some initial segment of $\dbar$. Therefore $\bbar^0 \models p^{\otimes |B \cap\cset{0}|}\res A \bbar^1$. Since $\bbar^1$ is a Morley sequence in $p$ over $A$, this implies that $\bbar^0 \bbar^1$ is a Morley sequence in $p$ over $A$.

  Finally, since we can do this for any $n$, we have that $T$ has $k$-ATP.
\end{proof}

\begin{thm}\label{thm:CTP-char}
  Fix a theory $T$. The following are equivalent.
  \begin{enumerate}
  \item\label{CTP-char-1} $T$ is NCTP.
  \item\label{CTP-char-2} For any set of parameters $A$ and formula $\varphi(x,b)$, if $\varphi(x,b)$ Kim-divides over $A$, then $\varphi(x,b)$ $p$-Kim-divides for every $A$-bi-invariant type $p(y)$ extending $\tp(b/A)$.
  \item\label{CTP-char-3} For any model $M$ and formula $\varphi(x,b)$, if $\varphi(x,b)$ $q$-Kim-divides for some $M$-coheir $q(y) \supseteq \tp(b/M)$, then $\varphi(x,b)$ $p$-Kim-divides for every $M$-heir-coheir $p(y) \supseteq \tp(b/M)$.
  \end{enumerate}
\end{thm}
\begin{proof}
  \cref{prop:Kim-failure-implies-CTP} gives that (\ref{CTP-char-1}) implies (\ref{CTP-char-2}), and it is immediate that (\ref{CTP-char-2}) implies (\ref{CTP-char-3}). By \cref{prop:countable-characterization-CTP} (or \cref{prop:char-uncountable-languages} if $\Lc$ is uncountable), we have that if $T$ has CTP, then (\ref{CTP-char-3}) fails. Therefore (\ref{CTP-char-3}) implies (\ref{CTP-char-1}).
\end{proof}
  





\begin{cor}\label{cor:NKL-to-NATP}
  If a theory $T$ satisfies new Kim's lemma (in the sense of \cite{NKL}), then it is NCTP (and therefore also NATP).
\end{cor}
\begin{proof}
  \cref{thm:CTP-char} implies that if $T$ has CTP, then there is a formula that Kim-divides over a model but does not Kim-divide with regards to some heir-coheir. Heir-coheirs are strictly invariant and therefore Kim-strictly invariant, so we have that $T$ fails to satisfy new Kim's lemma. Finally, since right-combs are a special kind of antichain, NCTP clearly implies NATP.
\end{proof}


One issue which we have been ignoring up until now is whether $k$-CTP implies $\ell$-CTP for $\ell < k$. This is an important structural property of SOP$_1$ and ATP. The proofs going into the proof of \cref{thm:CTP-char} clearly preserve the relevant degree of inconsistency. We have been unable to resolve this question, and likewise we have been unable to show that CTP is always witnessed by a formula in a single free variable. Despite the similarity between CTP and ATP, the proofs of these facts for ATP in \cite{Ahn2022} seem to rely pretty heavily on nice structural properties of antichains that right-combs do no share (e.g., the fact that an `antichain of antichains' is an antichain, which is used in \cite[Lem.~3.20]{Ahn2022}).

\begin{quest}
  If $T$ has CTP, does it follow that $T$ has $2$-CTP?
\end{quest}

\begin{quest}
  If $T$ has CTP, does it have CTP witnessed by a formula with a single free variable?
\end{quest}

One thing to note is that the proofs of these kinds of facts often make good use of indiscernible trees. As observed in \cite[Rem.~5.6]{Some-Remarks-Kim-dividing-NATP}, CTP is always witnessed by a strongly indiscernible tree. While this will almost certainly be an important tool for studying NCTP theories at some point, we nevertheless find it interesting that tree indiscernibility plays no role in any of the proofs in this paper.

\section{Reliably invariant types}
\label{sec:reliably}

\subsection{What should Kim-dividing over invariance bases be?}%
\label{sec:what-should}

In \cite{KaplanRamseyOnKim}, Kaplan and Ramsey originally defined Kim-dividing in terms of dividing along sequences generated by arbitrary invariant types. Pretty quickly in their analysis, however, it becomes clear that the natural concept is dividing along sequences generated by coheirs. Ramsey has in conversation consistently expressed the opinion that the definition of Kim-dividing in terms of coheirs is more natural. Coheirs satisfy two important properties that arbitrary invariant types do not:
\begin{itemize}
\item (\emph{Expansion}) If $p$ is an $M$-coheir and $M^\dagger$ is an expansion of $M$, then there is an $M^\dagger$-coheir $q$ extending $p$.
\item (\emph{Left extension}) If $p(x)$ is an $M$-coheir and $q(x,y)$ is a type over $M$ extending $p \res M$, then there is an $M$-coheir $r(x,y)$ extending $p(x)\cup q(x,y)$.
\end{itemize}
In the context of NSOP$_1$ theories (over models), the distinction between these two definitions becomes immaterial given the relevant Kim's lemma. Dividing along some invariant Morley sequence implies dividing along all invariant Morley sequences. In SOP$_1$ theories, however, a reasonable question, frequently asked by Kruckman, is whether the original definition of Kim-dividing is formula independent, i.e., if $\varphi(x,b)$ and $\psi(x,c)$ are logically equivalent and $\varphi(x,b)$ Kim-divides over a model $M$, then $\psi(x,c)$ Kim-divides over $M$. It follows from \cite[Thm.~3.10]{Some-Remarks-Kim-dividing-NATP} that in an NATP theory, this is always the case. In \cref{sec:Kim-dividing-extendibly}, we give an example showing that this can fail for ATP theories.

When attempting to define a robust notion of Kim-dividing over invariance bases, one ideally would like to retain the two nice properties of coheirs mentioned above. The issue with expansion, however, is that it characterizes coheirs over invariance bases.

\begin{prop}\label{prop:coheir-characterization}
  Fix a set of parameters $A$.
  \begin{enumerate}
  \item\label{if-invariant-acl-is-dcl} If $A$ is an invariance base,\footnote{Note that we are only considering invariance with regards to ordinary type, rather than Kim-Pillay or Lascar strong type, although over a set that is an invariance base in this sense, these notions collapse by essentially the same argument as in the proof of \ref{if-invariant-acl-is-dcl}.} then $\acl(A) = \dcl(A)$. 
  \item\label{invariant-tfae} Assume that $A = \acl(A)$. Fix an $A$-invariant type $p(x)$. The following are equivalent.
    \begin{enumerate}
    \item\label{invariant-in-expansions} For any model $M \supseteq A$ and any expansion $M^\dagger$ of $M$, $p\res M$ has a completion in $S(M^\dagger)$ that is $\Aut(M^\dagger/A)$-invariant.
    \item\label{dcl-sat} $p(x)$ is finitely satisfiable in $A$.
    \end{enumerate}
  \end{enumerate}
\end{prop}
\begin{proof}
  For \ref{if-invariant-acl-is-dcl}, suppose that $A$ is an invariance base. Fix an algebraic type $p(x) \in S(A)$. Since $A$ is an invariance base, $p(x)$ has an $A$-invariant global extension $q(x)$. Let $a_0,\dots,a_{n-1}$ be the realizations of $p$. Let $b$ be a realization of $q\res \acl(A)$. It must be the case that $b = a_i$ for some $i<n$, but by invariance, this implies that $n=1$.

  \ref{dcl-sat} $\To$ \ref{invariant-in-expansions}. If $p(x)$ is finitely satisfiable in $A$, then there is an ultrafilter $\Uc$ on $A$ whose average type is $p(x)$. The average type of $\Uc$ will be an invariant type extending $p(x)$ in any expansion of the theory.

  $\neg$\ref{dcl-sat} $\To$ $\neg$\ref{invariant-in-expansions}. Our proof of this is somewhat technical, so we have opted to put it in \cref{sec:coheir-inv-ext-exp}. %
\end{proof}

In particular, if $A$ is an invariance base, then every type over $A$ has an $A$-invariant extension satisfying \ref{invariant-in-expansions} if and only if $\dcl(A)$ is a model.

Fortunately, though, expansion seems to be more of a convenience than a necessity. The second property, however, seems to be fairly significant. All of this might suggest focusing on the following special class of invariant types.

\begin{defn} \label{defn:extendibly-invariant}
  A global type $p(x)$ is \emph{extendibly $A$-invariant} if for any $q(x,y) \in S(A)$ extending $p\res A$, $p(x) \cup q(x,y)$ extends to an $A$-invariant type. %
\end{defn}

As we already said, coheirs over models are always extendibly invariant. The example given in \cref{sec:Kim-dividing-extendibly} shows that not all invariant types over models are extendibly invariant.

While extendibly invariant types might seem relatively special, we get them for free over invariance bases. %

\begin{prop}\label{prop:extendible-nice}
  Let $A$ be a set of parameters.
  \begin{enumerate}
  \item\label{if-extendible-then-invariance-base} If there is an extendibly $A$-invariant type, then $A$ is an invariance base.
  \item \label{extendible-characterization} An $A$-invariant type $p(x)$ is extendibly $A$-invariant if and only if for every formula $\varphi(x,b) \in p(x)$ and every $A$-formula $\psi(x,y)$, if $\varphi(x,b)\wedge \psi(x,y)$ quasi-forks\footnote{Recall that a formula $\chi(x,c)$ \emph{quasi-forks} over $A$ if there is no $A$-invariant type containing $\chi(x,c)$. This is equivalent to implying a disjunction of formulas that quasi-divide over $A$.} over $A$, then $p(x) \vdash \neg \exists y \psi(x,y)$.
  \item\label{extendible-restriction} If $p(x,z)$ is extendibly $A$-invariant and $r(x)$ is the restriction of $p(x,z)$ to $x$, then $r(x)$ is extendibly $A$-invariant.
  \item\label{extendible-saturation} If $p(\xbar)$ is the type of an $(|A|+|T|)^+$-saturated model, then any $A$-invariant extension of $p(\xbar)$ is extendibly $A$-invariant.
  \item\label{if-invariance-base-then-extendible} If $A$ is an invariance base, then every type over $A$ extends to an extendibly $A$-invariant type.
  \item\label{extendibly-extendible} If $p(x)$ is an extendibly $A$-invariant type, then for any $q(x,y) \in S(A)$ extending $p\res A$, $p(x)\cup q(x,y)$ extends to an extendibly $A$-invariant type.
  \end{enumerate}
\end{prop}
\begin{proof}
    \ref{if-extendible-then-invariance-base} is immediate.  For \ref{extendible-characterization}, let $p(x)$ be an $A$-invariant type. Suppose that there is some formula $\varphi(x,b) \in p(x)$ and some $A$-formula $\psi(x,y)$ such that $\varphi(x,b)\wedge \psi(x,y)$ quasi-forks over $A$ and $p(x) \vdash \exists y \psi(x,y)$. Let $q(x,y)$ be a complete type extending $p\res A \cup \{\psi(x,y)\}$. We now have that $p(x)\cup q(x,y)$ has no $A$-invariant global extension, so $p(x)$ is not extendibly $A$-invariant.

  Now assume that $p(x)$ satisfies the condition in \ref{extendible-characterization}. Fix $q(x,y) \in S(A)$ extending $p\res A$. We have that for every $\varphi(x,b) \in p(x)$ and every $\psi(x,y) \in q(x,y)$, $\varphi(x,b) \wedge \psi(x,y)$ does not quasi-fork over $A$. By compactness, this implies that there is an $A$-invariant type $r(x,y)$ containing each of these formulas, which implies that $r(x,y) \supseteq p(x)\cup q(x,y)$.
  
  For \ref{extendible-restriction}, let $q(x,y)$ be a type over $A$ extending $r \res A$. $(p \res A)(x,z) \cup q(x,y)$ is consistent. Let $s(x,y,z)$ be a completion of $(p\res A)(x,z)\cup q(x,y)$. We now have that there is an $A$-invariant type $t(x,y,z)$ extending $p(x,z)\cup s(x,y,z)$. The restriction of $t$ to $xz$ is now the required type.

  For \ref{extendible-saturation}, let $r(\xbar)$ be an $A$-invariant extension of $p(\xbar)$. Note that \ref{extendible-characterization} implies that it is sufficient to check that every restriction of $r(\xbar)$ to finitely many variables is extendibly $A$-invariant. Let $s(x)$ be a restriction to finitely many variables. Let $q(x,y)$ be some type over $A$ extending $s \res A$. Let $M$ be a realization of $p(\xbar)$ and let $b$ be the finite tuple of elements corresponding to the variable $x$. By saturation, there is some $c \in M$ such that $bc \models q(x,y)$. Therefore, the restriction of $p(\xbar)$ to the variables corresponding to $bc$ witnesses that $s(x)\cup q(x,y)$ has an $A$-invariant extension. Since we can do this for any finite tuple of variables, we have that $r(\xbar)$ is extendibly $A$-invariant by \ref{extendible-characterization}.

  \ref{if-invariance-base-then-extendible} and \ref{extendibly-extendible} now follow immediately from \ref{extendible-restriction} and \ref{extendible-saturation}.
\end{proof}

Given \cref{prop:extendible-nice}, we propose that defining Kim-dividing over invariance bases in terms of extendibly invariant types is likely to be a more robust notion than Kim-dividing defined in terms of arbitrary invariant types. Our strongest evidence that this is a good definition (at least in the context of NCTP theories) is the results of the next section, specifically \cref{prop:reliable-main}. As we discuss at the end of \cref{sec:Kim-Reliable}, however, it unclear that this evidence is completely solid. Nevertheless, we find \cref{cor:Kim-fork-Kim-div} fairly compelling.


\subsection{Kim's lemma for reliably invariant types}
\label{sec:Kim-Reliable}

\begin{defn}
  A sequence $(b_i)_{i<n}$ is an \emph{invariant sequence over $A$} if $b_i \equiv_A b_j$ for each $i<j<n$ and $b_i \indi_A b_{<i}$ for each $i<n$.
\end{defn}

\begin{defn}\label{defn:reliable}
  Given a class of $A$-invariant types $\Ic$, an $A$-invariant type $p(x)$ is \emph{reliably in $\Ic$} if it is in the largest class $\Rc \subseteq \Ic$ satisfying that
  \begin{enumerate}
  \item\label{rel-0} for any $p(x,y) \in \Rc$, the restriction $p(x)$ is in $\Rc$,
  \item\label{rel-1} for any $p(x) \in \Rc$ and $q(x,y) \in S(A)$ extending $p \res A$, there is an $r(x,y) \in \Rc$ extending $p(x)\cup q(x,y)$,\footnote{This includes the particular case where $p$ is the unique $0$-type over the monster. In other words, we are implicitly requiring that for any $q(y) \in S(A)$, there is an $r(y) \in \Rc$ extending $q(y)$.} and
  \item\label{rel-2} for any $p(x_0) \in \Rc$ and $q(x_0,\dots,x_{n-1}) \in S(A)$ extending $(p \res A)(x_0)$, if $q(\xbar)$ is the type of an invariant sequence over $A$, then there is an $r(\xbar) \in \Rc$ extending $p(x_0)\cup \dots \cup p(x_{n-1}) \cup q(x_0,\dots,x_{n-1})$.
  \end{enumerate}
  If $\Ic$ is the class of all $A$-invariant types and $p(x)$ is reliably in $\Ic$, then we say that $p(x)$ is \emph{reliably $A$-invariant}. If $\Ic$ is the class of $A$-coheirs and $p(x)$ is reliably in $\Ic$, we say that $p(x)$ is a \emph{reliable $A$-coheir}.
\end{defn}

The class $\Rc$ exists by the Knaster-Tarski theorem (although it could be empty). Note that reliability clearly implies Kim-strictness. Given \cref{cor:reliability-is-reliable} and the fact that there are types over models with no strictly invariant extensions, we know that it does not always imply strictness.

\begin{prop}\label{prop:reliable-main}
  If there is a formula $\varphi(x,y)$, a reliably $A$-invariant type $p(y)$, and an extendibly $A$-invariant type $q(y)$ such that $p\res A = q\res A$ and $\varphi(x,y)$ $q$-Kim-divides but does not $p$-Kim-divide, then $T$ has CTP.
\end{prop}
\begin{proof}
  Suppose that $\varphi(x,y)$ is $k$-inconsistent on Morley sequences generated by $q$. For each $n<\omega$, we will build a certain configuration $(b_\sigma)_{\sigma \in 3^n}$ from which we can extract a $k$-CTP tree of height $n$. By compactness, we'll have that $T$ has $k$-CTP.

  We will build the arrays $(b_\sigma)_{\sigma \in 3^n}$ inductively. For $n=1$, let $b_0$ be any realization of $p\res A$. Then let $b_1 \models p \res Ab_0$ and $b_2 \models q \res Ab_0b_1$. Note that $b_0b_1b_2$ is an invariant sequence over $A$. Let $p_1(\xbar_1) = p(x)$. 

  Now suppose we are given the array $(b_\sigma)_{\sigma \in 3^{n+1}}$ and a type $p_{n+1}(\ybar_{n+1})$ satisfying the following induction hypotheses:
  \begin{itemize}
  \item The $3$-element sequence $\langle (b_{0\frown \sigma})_{\sigma \in 3^n},(b_{1\frown \sigma})_{\sigma \in 3^n},(b_{2\frown \sigma})_{\sigma \in 3^n} \rangle$ is invariant over $A$.
  \item $p_{n+1}(\ybar_{n+1})$ is reliably $A$-invariant.
  \item $(b_{1\frown \sigma})_{\sigma \in 3^n} \models p_{n+1} \res  A \cup(b_{0\frown \sigma})_{\sigma \in 3^n}$.
  \item $b_{\langle 2 \rangle^{n+1}}\models q \res A \cup(b_{i\frown\sigma})_{i<2,\sigma\in 3^n}$ (where $\langle 2 \rangle^{n+1}$ is the sequence of length  $n+1$  consisting entirely of $2$'s).
  \end{itemize}
  \begin{sloppypar}
    Let $c_{0\frown \sigma} = b_\sigma$ for each $\sigma \in 3^{n+1}$. Since $p_{n+1}(\ybar_{n+1})$ is reliably $A$-invariant and since $(c_{0\frown \sigma})_{\sigma \in 3^{n+1}} \models p_{n+1}\res A$, we can find a reliably $A$-invariant type $p_{n+2}(\ybar_{n+2})$ extending $p_{n+1}(\ybar_{n+1}) \cup \tp((c_{0\frown \sigma})_{\sigma \in 3^{n+1}}/A)$. Let $(c_{1\frown \sigma})_{\sigma \in 3^{n+1}}\models p_{n+1}\res A \cup (c_{0 \frown \sigma})_{\sigma \in 3^{n+1}}$. Then find $(c_{2 \frown \sigma})_{\sigma \in 3^{n+1}}$ such that $(c_{2\frown \sigma})_{\sigma \in 3^{n+1}}\ind^i_A (c_{i \frown \sigma})_{i<2,\sigma \in 3^{n+1}}$ and $c_{\langle 2 \rangle^{n+1}}\models q \res A\cup (c_{i \frown \sigma})_{i<2,\sigma \in 3^{n+1}}$ (which we can always do, since $q$ is extendibly $A$-invariant). Finally, let $(b_\sigma)_{\sigma \in 3^{n+2}}$ be $(c_\sigma)_{\sigma \in 3^{n+2}}$. 
  \end{sloppypar}
 By induction we can perform this procedure for as many steps as we like. Fix some $n$ and consider the cube $(b_\sigma)_{\sigma \in 3^n}$. Let $f$ be the map from $2^{< n}$ into $3^n$ taking $\tau \in 2^{<n}$ to the unique $\sigma \in 3^n$ satisfying that $\sigma(i)=0$ if and only if $\tau(i) = 1$, $\sigma(i)=1$ if and only if $\tau(i) = 0$, and $\sigma(i)=2$ if and only if $\tau(i)$ is not defined. Note that for any $\tau \in 2^{<n}$, $f(\tau)$ only has $2$'s at its end and either has a final segment of $2$'s or has no $2$'s at all.

 We would like to argue that $(b_{f(\tau)})_{\tau \in 2^{<n}}$ is a $k$-CTP tree of height $n$ (i.e., satisfies the definition of $k$-CTP except for the requirement that the tree have infinite height).

 \vspace{1em}

 \begin{claim}
   For any $\tau \in 2^{<n}$, $b_{f(\tau)} \models q \res A\cup (b_{f(\eta)})_{\eta \tgess \tau}$.
 \end{claim}
 \begin{claimproof}
 Let $\gamma$ be the longest initial segment of $f(\tau)$ not containing any $2$'s. Let $m$ be the number of $2$'s in $f(\tau)$. It follows by one of the induction hypotheses in the construction of the cube that $b_{f(\tau)}\models q \res A\cup (b_{\gamma\frown i \frown \e})_{i<2,\e \in 3^{m-1}}$. The family $(b_{\gamma\frown i \frown \e})_{i<2,\e \in 3^{m-1}}$ contains $b_{f(\eta)}$ for any $\eta \in 2^{<n}$ with $\eta \tgess \tau$, so the claim follows.
 \end{claimproof}

 \vspace{1em}

 So, in particular, we have that for any path $\eta \in 2^n$, $\{\varphi(x,b_{f(\eta \res i)}): i< n\}$ is $k$-inconsistent.

 \vspace{1em}

 \begin{claim}
   For any $\tau \in 2^{< n}$ and $\eta \tgeq \tau \frown 0$, $b_{f(\eta)} \models p \res A \cup (b_{f(\gamma)})_{\gamma \tgeq \tau \frown 1}$. 
 \end{claim}
 \begin{claimproof}
   Let $\delta$ be the longest initial segment of $f(\tau)$ not containing any $2$'s. We then have that $f(\gamma)$ has $\delta \frown 0$ as an initial segment for any $\gamma \tgeq \tau\frown 1$ (since $f$ switches $0$ and $1$). Let $m$ be the number of $2$'s in $f(\tau)$. By the construction of the cube, we have that for any $\e$ with initial segment $\delta \frown 1$, $b_\e \models p \res A \cup (b_{\delta\frown 0 \frown \xi})_{\xi \in 3^{m-1}}$. The family $(b_{\delta\frown 0 \frown \xi})_{\xi \in 3^{m-1}}$ includes $b_{f(\gamma)}$ for any $\gamma \tgeq \tau \frown 1$, so the claim follows.
 \end{claimproof}

 \vspace{1em}

 This last claim implies that if $X \subseteq 2^{\leq n}$ is a right-comb, then it is a Morley sequence generated by $p(y)$ when enumerated in decreasing lexicographic order. In particular, $\{\varphi(x,b_{f(\tau)}):\tau \in X\}$ is consistent for any right-comb $X$.

 Therefore, by compactness, $T$ has $k$-CTP.\footnote{We could also just build the $k$-CTP tree directly with part of a cube of the form $3^\omega$ (taken as a direct limit of the cubes $3^n$), but this is a bit more annoying to actually write out.}
\end{proof}

One notable thing about the above proof is that, despite the advocacy for the concept of extendibly invariant types given in \cref{sec:what-should}, extendibility is playing a relatively minor role in the proof. We never use the first or second property in \cref{defn:reliable} for the type $p(y)$ and extendibility for $q(y)$ seems to barely matter in that we throw away all of the `siblings' of the realization of $q$ that we build at each step. This raises the question of whether the restriction to extendibly invariant types is necessary. 

\begin{quest}
  Say that an $A$-invariant type is \emph{semi-reliably $A$-invariant} if it satisfies \cref{defn:reliable} with the first and second conditions removed (and with $\Ic$ the class of all $A$-invariant types). Is it true that if there is a set of parameters $A$, a formula $\varphi(x,y)$, a semi-reliably $A$-invariant type $p(y)$, and an $A$-invariant type $q(y)$ such that $p\res A = q \res A$ and $\varphi(x,y)$ $q$-Kim-divides but does not $p$-Kim-divide, then $T$ has CTP?
\end{quest}

Another interesting and frustrating thing is that, while reliably $A$-invariant types always exist over extension bases (as we will show in the next section), the construction of reliably invariant types in \cref{cor:reliability-is-reliable} and the construction of bi-invariant types in \cref{prop:countable-characterization-CTP} seem to be rather incompatible. \cref{cor:reliability-is-reliable} relies heavily on knowing the precise type over the set $A$, whereas \cref{prop:countable-characterization-CTP} is a forcing construction that explicitly needs to avoid committing to a specific complete type before the end. It's not even clear that the bi-invariant types constructed in \cref{prop:countable-characterization-CTP} are \emph{extendibly bi-invariant} (i.e., have the property that for any type $q(x,y)$ over $A$ extending $p\res A$, there is an $A$-bi-invariant type $r(x,y)$ extending $p(x) \cup q(x,y)$). 

\begin{quest}
  Does the converse of \cref{prop:reliable-main} hold? Moreover, if $T$ has CTP, does there exist a reliable heir-coheir witnessing this? A reliably bi-invariant type? Extendibly bi-invariant?
\end{quest}

It's far from clear that the definition we've given in \cref{defn:reliable} is in some sense the `correct' one. Its motivation is primarily that it is the strongest property that we can guarantee over invariance bases. In particular, while coheirs are always extendibly invariant, it is not clear whether heir-coheirs are always reliably invariant. If this were true we would be able to combine \cref{thm:CTP-char} and \cref{prop:reliable-main} and give a statement of Kim's lemma that both characterizes NCTP and is non-trivial over arbitrary extension bases (and therefore over arbitrary models).

\begin{quest}
  Are all heir-coheirs extendible heir-coheirs? In other words, if $p(x)$ is an $M$-heir-coheir and $q(x,y)$ is a type over $M$ extending $p \res M$, is there an heir-coheir extending $p(x) \cup q(x,y)$?
\end{quest}

\begin{quest}
  Are all heir-coheirs reliable coheirs? Reliably invariant?
\end{quest}


\subsection{Existence of reliably invariant types}%
\label{sec:exist-reli-invar}

Our approach to building reliably invariant types is to build invariant types that are `as bi-invariant as possible.' In order to manage our bookkeeping, we will use a fixed tuples of variables indexed by the monster model.

\begin{defn}
  Let $\Ob$ be a fixed monster model of $T$. Let $\Xb = (x_a)_{a \in \Ob}$ be a fixed enumeration of distinct variables indexed by $\Ob$. Given a formula $\varphi(\Xb)$ (with parameters in the monster) and $\sigma \in \Aut(\Ob)$, we write $\sigma^{\varr}\cdot \varphi$ to represent the formula $\varphi(\Xb)$ with its variables permuted by $\sigma$ in the obvious way (but not its parameters). We write $\sigma^{\parr}\cdot\varphi$ to represent the formula $\varphi(\Xb)$ with its parameters permuted by $\sigma$ (but not its variables). We will use the same notation for partial types to indicate the image under the corresponding map.

  The \emph{minimal monster type (over $M$)}, $\Theta_M(\Xb)$, is $\tp(\Ob/M)$ with the obvious variable assignment. %
  A \emph{monster type (over $M$)} is any consistent type extending $\Theta_M(\Xb)$.

  A partial type $\Sigma(\Xb)$ is \emph{$A$-invariant} if for any $\sigma \in \Aut(\Ob/A)$, $\sigma^{\parr}\cdot \Sigma = \Sigma$. $\Sigma(\Xb)$ is \emph{$A$-co-invariant} if for any $\sigma \in \Aut(\Ob/A)$, $\sigma^{\varr}\cdot \Sigma= \Sigma$.%
\end{defn}

Note in particular that the minimal monster type over $M$ is $M$-invariant (trivially) and $M$-co-invariant. Also note that (the deductive closure of) the union of any two $A$-invariant types is $A$-invariant and of any two $A$-co-invariant types is $A$-co-invariant.

When $\abar$ is a tuple of elements of the monster, we may write $x_{\abar}$ to represent the corresponding tuple of variables in $\Xb$.

\begin{lem}\label{lem:parallel-transport}
  Fix a set $A$. Let $\Xi(\Xb)$ be a maximal\/\footnote{Whenever we talk about maximal partial types among some class, we mean maximal consistent partial types among that class.} $A$-co-invariant partial type. For any tuple $\abar \in \Ob$ and any formula $\varphi(x_{\abar})$, if $\varphi(x_{\abar})$ is consistent with $\Xi(\Xb)$, then there are $\abar_0,\dots,\abar_{n-1} \equiv_A \abar$ such that $\Xi(\Xb) \vdash \bigvee_{i<n}\varphi(x_{\abar_i})$.
\end{lem}
\begin{proof}
  Suppose that no such $\abar_0,\dots,\abar_{n-1}$ exist. Then we'd have that $\{\Xi(\Xb)\cup\{\neg \varphi(x_{\abar'}) : \abar' \equiv_A \abar\}$ is consistent and $A$-co-invariant, but then by maximality, we'd have that $\Xi(\Xb)\vdash \neg \varphi(x_{\abar})$, contradicting the fact that $\varphi(x_{\abar})$ is consistent with $\Xi(\Xb)$.
\end{proof}

\begin{lem}\label{lem:max-does-not-fracture}
  Let $\Xi(\Xb)$ be a maximal $A$-co-invariant partial type, and let $\abar_0,\dots,\allowbreak\abar_{n-1}$ be an invariant sequence over $A$.
  \begin{enumerate}
  \item If a formula $\varphi(x_{\abar_0})$ is consistent with $\Xi(\Xb)$, then $\bigwedge_{i<n}\varphi(x_{\abar_i})$ is consistent with $\Xi(\Xb)$.
  \item If a global type $p(x_{\abar_0})$ is consistent with $\Xi(\Xb)$, then $\bigwedge_{i<n}p(x_{\abar_i})$ is consistent with $\Xi(\Xb)$.
  \end{enumerate}
\end{lem}
\begin{proof}
  For 1, since $\varphi(x_{\abar_0})$ is consistent with $\Xi(\Xb)$, we have by \cref{lem:parallel-transport} that there are $\bbar_0,\dots,\bbar_{k-1}$ such that $\Xi(\Xb)\vdash \bigvee_{j<k}\varphi(\bbar_i)$ and $\bbar_j \equiv_A \abar_0$ for each $j<k$.

  For each positive $i<n$, let $p(\xbar)$ be an $M$-invariant type such that $\abar_i\models p \res A \abar_{<i}$.

  By induction, we can build a forest $(\dbar_\eta)_{\eta \in k^{\leq n}\setminus \{\varnothing\}}$ such that
  \begin{itemize}
  \item for each $\eta \in k^{< n}$, $\dbar_{\eta \frown 0}\dbar_{\eta \frown 1}\dots \dbar_{\eta \frown (n-1)} \equiv_M \cbar_0\cbar_1 \dots \cbar_{n-1}$ and
  \item for each $\eta \in k^{< n}\setminus \{\varnothing\}$, $\dbar_\eta \models p_{n-\ell(\eta)} \res \dbar_{\tgess \eta}$ (where $\ell(\eta)$ is the length of $\eta$).
  \end{itemize}
  Let $q(\Xb)$ be a completion of $\Xi(\Xb)$. Since $\Xi(\Xb)\vdash \bigvee_{j<k}\varphi(x_{\dbar_{\eta\frown j}})$ for each $\eta \in k^{<n}$, we can find a path $\beta \in k^n$ such that $q(\Xb)\vdash \bigwedge_{0<i\leq n}\varphi(x_{\dbar_{\beta \res i}})$. This implies that $\bigwedge_{0<i\leq n}\varphi(x_{\dbar_{\beta \res i}})$ is consistent with $\Xi(\Xb)$, so by $A$-co-invariance, $\bigwedge_{i<n}\varphi(x_{\abar_i})$ is consistent with $\Xi(\Xb)$ as well.

  2 follows by compactness.
\end{proof}

The following proposition says that we can rely on the existence of reliable types.
\begin{thm}\label{thm:I-reliable-existence}
  Fix a set of parameters $A$ and let $\Ic$ be a class of $A$-invariant types. Suppose furthermore that
  \begin{enumerate}
  \item\label{I-extension} every type over $A$ extends to a member of $\Ic$,
  \item\label{I-closure} for each tuple of variables $\xbar$, $S_{\xbar}(\Ob)\cap \Ic$ is closed, 
  \item\label{I-permutation} $\Ic$ is invariant under renaming variables, and
  \item\label{I-restriction} any $A$-invariant type $p(\xbar)$ is in $\Ic$ if and only if every restriction of it to finitely many variables is in $\Ic$.
  \end{enumerate}
  Then for every $q(\xbar) \in S(A)$, there is an $A$-invariant $p(\xbar)\supset q(\xbar)$ that is reliably in $\Ic$.
\end{thm}
\begin{proof}
  Let $F \subseteq S_{\Xb}(\Ob)$ be the set of extensions of $\Theta_A(\Xb)$ that are in $\Ic$. Note that since $F$ is the intersection of two closed sets, it is closed. Let $\Sigma(\Xb)$ be the global partial type corresponding to the closed set $F$.

  \vspace{1em}

  \begin{claim}
    $\Sigma(\Xb)$ is $A$-co-invariant.
  \end{claim}
  \begin{claimproof}
    Fix $\sigma \in \Aut(\Ob/A)$. By \ref{I-permutation} and \ref{I-restriction}, $\Ic \cap S_{\Xb}(\Ob)$ is fixed by the action of $\sigma$ on the variables $\Xb$. Since the set of global completions of $\Theta_A(\Xb)$ is as well, this implies that $F$ and therefore $\Sigma(\Xb)$ is fixed under the action of $\sigma$ on $\Xb$. Since we can do this for any $\sigma$, $\Sigma(\Xb)$ is $A$-co-invariant.
  \end{claimproof}

  \vspace{1em}

  Let $\Xi(\Xb)$ be a maximal $A$-co-invariant extension of $\Sigma(\Xb)$ (which always exists by Zorn's lemma). Note that $\Xi(\Xb)$ is also $A$-invariant.\footnote{Normally we couldn't require a maximal $A$-co-invariant partial type to be $A$-invariant, but since every complete type in $F$ is $A$-invariant, \emph{any} partial type extending $\Sigma(\Xb)$ is $A$-invariant.} %
  Also note that since $\Xi(\Xb)$ is a monster type, we have that if $a$ is a realization of $p(x) \in S(A)$, then $\Xi(\Xb) \vdash p(x_a)$.

  Now let $\Sc$ be the class of $A$-invariant types $p(\xbar)$ satisfying the following property:
  \begin{itemize}
  \item[$(\ast)$] For any finite sub-tuple $y_0,\dots,y_{n-1}$ of $\xbar$, let $q(\ybar)$ be the restriction of $p$ to the variables $\ybar$. If $a_0a_1\dots a_{n-1} \models q \res A$, then $q(x_{a_0},x_{a_1},\dots,x_{a_{n-1}}) \cup \Xi(\Xb)$ is consistent.
  \end{itemize}
Note that since $\Xi(\Xb)$ is $A$-co-invariant, the above property doesn't depend on the choice of $\abar$. We would like to show that the class $\Sc$ satisfies the closure properties in \cref{defn:reliable} relative to $\Ic$ and therefore every type in $\Sc$ is reliably in $\Ic$.

The first closure property is immediate. Note that for small tuples of variables, the second closure property is also immediate: If $p(\xbar)$ is an $A$-invariant type in $\Sc$ and $q(\xbar,\ybar)$ is a type in $S_{\xbar\ybar}(A)$ extending $p\res A$, then given $\abar \bbar \models q$, we have that $\Xi(\Xb) \vdash q(x_{\abar},x_{\bbar})$ and so $p(x_{\abar})\cup q(x_{\abar},x_{\bbar})$ is automatically consistent (since $p(x_{\abar})$ is consistent with $\Xi(\Xb)$ by $(\ast)$). For types with large tuples of variables, the result follows from the fact that for a fixed type $q(\xbar)$ over $A$, the set of global extensions of $q(\xbar)$ satisfying $(\ast)$ is closed and so the required result follows by compactness.

Now we need to show that $\Sc$ satisfies the third closure property. Again, we first show this for small tuples of variables.

  \vspace{1em}

  \begin{claim}
    For any $p(\xbar) \in S(A)$ (with $\xbar$ a small tuple of variables) and invariant sequence $(\abar_i)_{i<n}$ over $A$, if $\abar_0 \models p$ and $q(x_{\abar_0})$ is a global extension of $p(x_{\abar_0})$ such that $\Xi(\Xb) \cup q(x_{\abar_0})$ is consistent, then $\Xi(\Xb)\cup q(x_{\abar_0})\cup q(x_{\abar_1})\cup \dots \cup q(x_{\abar_{n-1}}) \cup r(x_{\abar_0},\dots,x_{\abar_{n-1}})$ is consistent, where $r(x_{\abar_0},\dots,x_{\abar_{n-1}}) = \tp(\abar_0\dots \abar_{n-1}/A)$. 
  \end{claim}
  \begin{sloppypar}
    \begin{claimproof}
      Since $(\abar_i)_{i<n}$ is an invariant sequence over $A$, we have by \cref{lem:max-does-not-fracture} that $\Xi(\Xb)\cup q(x_{\abar_0})\cup q(x_{\abar_1})\cup\dots \cup q(x_{\abar_{n-1}})$ is consistent. Since $\Xi(\Xb)$ is a monster type, we have that $\Xi(\Xb)\vdash r(x_{\abar_0},\dots,x_{\abar_{n-1}})$, so the required partial type is consistent.
    \end{claimproof}
  \end{sloppypar}
  \vspace{1em}

  Therefore if $q(x_{\abar_0})$ is any global completion of $p(x_{\abar_0}) = \tp(a_0/A)$ that is consistent with $\Xi(\Xb)$ and $(\abar_i)_{i<n}$ is an invariant sequence over $A$, we have that for any $\bbar$, if $\abar_0 \models q \res A \bbar$, then we can find a global completion $s(x_{\abar_0},\dots,x_{\abar_{n-1}})$ consistent with $\Xi(\Xb)$ and $\abar'_1\dots \abar'_{n-1}\equiv_{A\abar_0}\abar_1\dots \abar_{n-1}$ such that $\abar_0\abar'_1\dots \abar'_{n-1} \models s \res A\bbar$.

  Again the third closure condition for types with large tuples of variables follows from compactness and the fact that if $(\abar_i\bbar_i)_{i<n}$ is an $A$-invariant sequence, then $(\abar_i)_{i<n}$ is an $A$-invariant sequence as well.

  Therefore $\Sc$ satisfies the closure conditions in \cref{defn:reliable} relative to $\Ic$ and so $\Sc \subseteq \Rc$, since $\Rc$ is the largest class of types satisfying this condition. Finally, if $q(\xbar)$ is a type in a small number of variables, we can find a global $A$-invariant type $p(\xbar) \supseteq q(\xbar)$ in the class $\Sc$, which is therefore reliably in $\Ic$. For types with large tuples of variables, the extension again exists by compactness.
\end{proof}

\begin{cor}\label{cor:reliability-is-reliable}
  Fix a set of parameters $A$.
  \begin{enumerate}
  \item If $A$ is an invariance base, then every type $p(x) \in S_x(A)$ extends to a reliably $A$-invariant type.
  \item If $A$ is a model, then every type $p(x) \in S_x(A)$ extends to a reliable $A$-coheir.
  \end{enumerate}
\end{cor}
\begin{proof}
  If $A$ is an invariance base, we can apply \cref{thm:I-reliable-existence} with $\Ic$ equal to the class of $A$-invariant global types. If $A$ is a model, we can apply \cref{thm:I-reliable-existence} with $\Ic$ equal to the class of $A$-coheirs.
\end{proof}

The following corollary is similar to \cite[Cor.~4.9.1]{Mutchnik-NSOP2}, although at the moment it isn't clear if they are directly comparable.

\begin{cor}\label{cor:Kim-fork-Kim-div}
 If $T$ is NCTP, then a formula $\varphi(x,b)$ Kim-forks with regards to extendibly invariant types over an invariance base $A$ if and only if it Kim-divides with regards to extendibly invariant types over $A$.
\end{cor}
\begin{proof}
  Suppose that $\varphi(x,b)$ Kim-forks over $A$. Let $\varphi(x,b)\vdash \bigvee_{i<n}\psi_i(x,c_i)$ such that $\psi_i(x,c_i)$ Kim-divides over $A$ for each $i<n$. Let $p(y,z_0,\dots,z_{n-1})$ be a reliably $A$-invariant type extending $\tp(bc_0\dots c_{n-1}/A)$. Let $(d^je_i^j)_{i<n,j<\omega}$ be a Morley sequence generated by $p$. Note that $d^{<\omega}$ and $e^{<\omega}_i$ for each $i<n$ are Morley sequences in reliably $A$-invariant types. By \cref{prop:reliable-main}, we have that $\{\psi_i(x,e_i^j):j<\omega\}$ is inconsistent for each $i<n$. Therefore, by the standard argument, we have that $\{\bigvee_{i<n}\psi_i(x,e_i^j):j<\omega\}$ is inconsistent. By indiscernibility, this implies that $\{\varphi(x,d^j):j<\omega\}$ is inconsistent, or, in other words, that $\varphi(x,b)$ Kim-divides over $A$ with regards to a reliably $A$-invariant type. Any reliably $A$-invariant type is extendibly $A$-invariant, so we are done.
\end{proof}




\section{Local Character}
\label{sec:local-char}

\subsection{Uncountable languages and dual local character}

In the following we will give a characterization of NCTP in terms of a weak dual local character under the assumption that there is a measurable cardinal $\kappa > |T|$. We don't expect this to be necessary, but given the unclear usefulness of this characterization and how straightforward the proof with a measurable cardinal is, we have decided to not pursue a stronger result for the time being. In the next proposition we will also extend \cref{prop:countable-characterization-CTP} to uncountable languages.

\begin{prop}\label{prop:char-uncountable-languages}
  Let $T$ be a theory in a (possibly uncountable) language $\Lc$. Assume $T$ has $k$-CTP witness by the formula $\varphi(x,y)$.
  \begin{enumerate}
  \item\label{char-uncountable-1} There is a model $M$ with $|M| = |\Lc|$, an $M$-heir-coheir $p(y)$, and an $M$-coheir $q(y)$ such that $p\res M = q \res M$ and $\varphi(x,y)$ $q$-Kim-divides but does not $p$-Kim-divide.
  \item\label{char-uncountable-2} For any regular cardinal $\kappa > |\Lc|$, there is a model $M$ with $|M| = \kappa$, a continuous chain $(N_\e)_{\e < \kappa}$ of elementary substructures with $M = \bigcup_{\e < \kappa}N_\e$ and $|N_\e|<\kappa$ for each $\e < \kappa$, and a sequence $(e_\e)_{\e < \kappa}$ such that $e_\e \notin N_\e$ for each $\e < \kappa$, $\{\varphi(x,e_\e):\e < \kappa\}$ is $k$-inconsistent, and for each $\e < \kappa$, there is an $N_\e$-heir-coheir $s_\e \supseteq \tp(e_\e/N_\e)$ such that $\varphi(x,e_\e)$ does not $s_\e$\nobreakdash-\hspace{0pt}Kim-divide.
  \end{enumerate}
\end{prop}
\begin{proof}
  The proofs of 1 and 2 are nearly the same. If we are proving 1, let $\kappa = |\Lc|$, and if we are proving 2, let $\kappa$ be as in the statement of the proposition. When we say a model $N$ is \emph{small}, we mean that $|N| \leq \kappa$ if we are proving 1 and $|N| < \kappa$ if we are proving $2$. 
  
  Let $(b_\sigma)_{\sigma \in 2^{<\omega}}$ witness that the formula $\varphi(x,y)$ has $k$-CTP. Let $M_0 \supseteq \{b_\sigma : \sigma \in 2^{<\omega}\}$ be a small model. Expand $M_0$ with the following predicates:
  \begin{itemize}
  \item A unary predicate $R_0$, which defines the set $\{b_\sigma:\sigma \in 2^{<\omega}\}$.
  \item A unary predicate $X_0$, which defines the set $\{b_\sigma : \sigma \tgeq 1\}$.
  \item A partial order $\tleq$ on $R_0$ defined by $b_\sigma \tleq b_\tau$ if and only if $\sigma \tleq \tau$.
  \item The functions $b_\sigma \mapsto b_{\sigma \frown 0}$ and $b_\sigma \mapsto b_{\sigma \frown 1}$, which we will write as $b \frown 0$ and $b \frown 1$.
  \end{itemize}
  Let $\Lc_0$ be the expanded language. %
  
  Let $(\beta(\alpha),\gamma(\alpha))_{\alpha < \kappa}$ be an enumeration of $\kappa^2$ with the property that for each pair $\delta,\e < \kappa$, the set $\{\alpha < \kappa : \langle \beta(\alpha),\gamma(\alpha) \rangle = \langle \delta,\e \rangle\}$ is cofinal in $\kappa$. (Such an enumeration can be defined using any bijection between $\kappa$ and $\kappa^3$.)

  Let $(R_\alpha)_{\alpha < \kappa}$ and $(X_\alpha)_{\alpha < \kappa}$ be two sequences of distinct unary predicates. For each $\alpha < \kappa$, let $\Lc_\alpha = \Lc_0 \cup \{R_\beta,X_\beta:\beta \leq \alpha\}$. We will inductively build a sequence of elements $(b_\alpha)_{\alpha < \kappa}$ and a chain of small structures $(M_\alpha)_{\alpha<\kappa}$ (with $M_\alpha$ an $\Lc_\alpha$\nobreakdash-\hspace{0pt}structure) where for any $\beta < \alpha$, $M_\beta$ is an $\Lc_\beta$-elementary substructure of $M_\alpha$. Each $R_\alpha$ will be downwards closed under $\tleq$ and closed under the functions $b\mapsto b\frown 0$ and $b\mapsto b\frown 1$. We will call any set closed under these a \emph{closed tree}. %
  Furthermore, for $\beta < \alpha$, we will have $R_\alpha \subseteq R_\beta$, $X_\alpha \subseteq X_\beta$, and $X_\alpha \subseteq R_\alpha$. (If we are proving $1$, the $X_\alpha$'s will eventually be part of the filter that we use to define our two coheirs.) %

  Let $c_0 = b_\varnothing$. At stage $\alpha < \kappa$, suppose we have our small $\Lc_\alpha$-structure $M_\alpha$ (i.e., if we are proving 1, then $|M_\alpha| \leq \kappa$ and if we are proving 2 then $|M_\alpha| < \kappa$). Furthermore, suppose that $X_\alpha$ is dense above $c_\alpha\frown 1$ in $R_\alpha$ (i.e., for any $c \in R_\alpha(M_\alpha)$ with $c \tgeq b_\alpha$, there is a $d \in R_\alpha(M_\alpha)$ such that $d \tgeq c$ and $d \in X_\alpha(M_\alpha)$). Let $(\psi^\alpha_{\beta(\delta)}(y,z))_{\delta < \kappa}$ be an enumeration of all $\Lc_\alpha(M_\alpha)$-formulas. To get $M_{\alpha+1}$ and $c_{\alpha+1}$, perform the following construction:
  \begin{itemize}
  \item If $\psi^{\gamma(\alpha)}_{\beta(\alpha)}(y,z)$ has already been defined and there is an $\Lc_\alpha$-elementary extension $N \succeq M_\alpha$, a closed tree $S \subseteq R_\alpha(N)$ with $S \supseteq R_\alpha(M_\alpha)$, and a $d$ (in the $\Th(M_\alpha)$-monster) such that $X_\beta(N)$ is dense above $c_{\beta+1}$ in $S$ for each $\beta < \alpha$ and $\{b \in S\cap X_\alpha(N) : \psi^{\gamma(\alpha)}_{\beta(\alpha)}(b,d)\}$ is somewhere dense above $c_\alpha \frown 1$ in $S$: By L\"owenheim-Skolem, we can assume that $N$ is small. Let $M_{\alpha+1}\supseteq N d$ be a small elementary extension. Find $c_{\alpha+1} \in X_{\alpha}(N)$ with $c_{\alpha+1} \tgeq c_\alpha \frown 1$ such that $\psi^{\gamma(\alpha)}_{\beta(\alpha)}(c_{\alpha+1},d)$ holds. By compactness, we may assume that for any $b \in R_\alpha(M_\alpha)$, $\neg(b\tgeq c_{\alpha+1})$. Let $R_{\alpha+1}(M_{\alpha+1}) = S$, and let
    \[
      X_{\alpha+1}(M_{\alpha+1})=\{b \in S\cap X_\alpha(M_{\alpha+1}) : b \tgeq c_{\alpha+1}\frown 1,~\psi^{\gamma(\alpha)}_{\beta(\alpha)}(b,d)\}.
    \]
    Note that $X_{\alpha+1}(M_{\alpha+1})$ is dense above $c_{\alpha+1}\frown 1$ in $R_{\alpha+1}(M_{\alpha+1})$.
  \item If the condition in the previous bullet point fails: Find a small $\Lc_\alpha$-elementary extension $M_{\alpha+1}\succeq M_\alpha$ with a $c_{\alpha+1}\tgeq c_\alpha\frown 1$ such that $c_{\alpha+1}\in X_\alpha(M_{\alpha+1})$ and such that for any $b \in R_\alpha(M_{\alpha})$, $\neg(b \tgeq c_{\alpha+1})$. (This is always possible because $R_\alpha(M_{\alpha})$ is a closed tree.) Let $R_{\alpha+1}(M_{\alpha+1})=R_\alpha(M_{\alpha+1})$, and let $X_{\alpha+1}(M_{\alpha+1}) = \{b \in X_{\alpha}(M_{\alpha=1}) : b \tgeq c_{\alpha+1}\frown 1\}$.
  \end{itemize}
  Note that in both cases we have ensured that $c_{\alpha+1} \in X_{\alpha}$ and that $c_{\alpha+1}$ is not $\tleq$-upper-bounded by any element of $R_\alpha(M_\alpha)$. Also note that for every $\beta \leq \alpha$, we have that $X_\beta(M_{\alpha+1})$ is dense above $c_{\beta+1}$ in $R_{\alpha+1}(M_{\alpha+1})$.

  Just before limit stage $\alpha < \kappa$, suppose that we have $M_\beta$ and $c_{\beta}$ for all $\beta < \alpha$. Let $N = \bigcup_{\beta < \alpha}M_\beta$, which we can regard as an $\Lc_{<\alpha}$-structure (where $\Lc_{<\alpha}\coloneqq \bigcup_{\beta<\alpha}\Lc_\beta$). By compactness, we can find a small model $M_{\alpha} \succeq N$, a closed tree $R_\alpha(M_\alpha)\subseteq M_\alpha$, a set $X_\alpha(M_\alpha)\subseteq R_\alpha(M_\alpha)$, and $c_\alpha \in R_\alpha(M_\alpha)$ such that for all $\beta < \alpha$, $R_\alpha(M_\alpha)\subseteq R_\beta(M_\alpha)$, $X_\alpha(M_\alpha)\subseteq X_\beta(M_\alpha)$, $c_\alpha \tgeq c_\beta$, and $X_\alpha(M_\alpha)$ is dense above $c_\alpha \frown 1$ in $R_\alpha(M_\alpha)$. If we are proving 2 and $\{X_\beta(y):\beta < \alpha\}$ axiomatizes a complete $\Lc_{<\alpha}$\nobreakdash-\hspace{0pt}type $r_\alpha(y)$ over $N$, we may moreover assume that there is a $d_\alpha \in M_\alpha$ realizing $r_\alpha(y)$ such that $c_\beta \tleq d_\alpha \tleq c_\alpha$ for all $\beta < \alpha$. 

  After the construction is completed, let $M = \bigcup_{\alpha < \kappa}M_\alpha$. Note that $|M|=\kappa$ by induction. 

  \vspace{1em}

  \noindent \emph{Proof of 1.} Let $R = \bigcup_{\alpha < \kappa}R_\alpha(M_\alpha)$. Note that $R$ is a closed tree. %

  Let $\Fc$ be the filter generated by $\{X_\alpha(M):\alpha < \kappa\}\cup \{Y\}$, where $Y \coloneqq \bigcup_{\alpha<\kappa}\{b \in R : b \tgeq c_\alpha \frown 0\}$.

  \vspace{1em}

   \noindent \emph{Claim 1.} For every $\beta < \kappa$, $X_\beta(M)$ is dense above $c_{\beta+1}$ in $R$.
  \begin{claimproof}
    Fix a $\beta < \kappa$. At each stage $\alpha < \kappa$ with $\alpha \geq \beta$, we ensured that $X_\beta(M_{\alpha+1})$ is dense above $c_{\beta+1}$ in $R_{\alpha+1}(M_{\alpha+1})$. Fix $b \in R$ with $b \tgeq c_{\beta+1}$. Since $R$ is the union of a chain, there is an $\alpha < \kappa$ with $\alpha \geq \beta$ such that $b \in R_{\alpha+1}(M_{\alpha+1})$. At stage $\alpha$, we ensured that $ X_\beta(M_{\alpha+1})$ is dense above $c_{\beta+1}$. Therefore there is a $d \in X_\beta(M_{\alpha+1})\cap R_{\alpha+1}(M_{\alpha+1})$ with $d\tgeq b$. Since $R_{\alpha+1}(M_{\alpha+1}) \subseteq R$, we have that $d \in R$. By elementarity, we also have that $b \in X_{\beta}(M)$. Since we can do this for any $b \in R$ with $b \tgeq c_{\beta+1}$, we have that $X_\beta(M)$ is dense above $c_{\beta+1}$ in $R$.
  \end{claimproof}

  \vspace{1em}

  The previous claim implies that the filter $\Fc$ is everywhere somewhere dense. Therefore we can extend it to an everywhere somewhere dense ultrafilter $\Uc$ by the same argument as in \cref{lem:everything-everywhere-all-at-once}.

  Let $\Gc$ be the filter generated by $\{X_\alpha(M): \alpha < \kappa\}\cup\{\{c_\alpha : \alpha < \kappa\}\}$. Since $c_\alpha \in X_{\beta}(M_\beta)\subseteq X_{\beta}(M)$ for any $\beta < \alpha < \kappa$, we have that $\Gc$ is a proper filter. Let $\Vc$ be an ultrafilter extending $\Gc$.

  Let $p(y)$ be the global average type of $\Uc$, and let $q(y)$ be the global average type of $\Vc$.

  \vspace{1em}

   \noindent \emph{Claim 2.} $p \res M = q\res M$. Moreover, this type is axiomatized by $\{X_\alpha(y): \alpha < \kappa\}$.
  \begin{claimproof}
    Fix an $M$-formula $\psi(y) \in p(y)$. $\psi(y)$ is actually an $M_\beta$ formula for some $\beta < \kappa$. By the choice of our enumeration $(\beta(\alpha),\gamma(\alpha))_{\alpha < \kappa}$, this means that there was a stage $\alpha$ at which $\psi^{\beta(\alpha)}_{\gamma(\alpha)}(y,z)$ was defined and equal to $\psi(y)$ (where $z$ is a dummy variable). Since $\psi(y) \in p(y)$ and since $\Uc$ is everywhere somewhere dense, we must have chosen the first bullet point at this stage. (In particular, the required condition is satisfied with $N = M$ and $S=R$.) Hence $\psi(b)$ holds for all $b \in X_{\alpha+1}(M_{\alpha+1})$. By elementarity, this implies that $\psi(b)$ holds for all $b \in X_{\alpha+1}(M)$ as well, implying that $\psi(c_{\delta})$ holds for all $\delta > \alpha + 1$. Therefore $\psi(y) \in q(y)$ as well. Since we can do this for every $M$-formula $\psi(y)\in p(y)$, we have that $p \res M = q \res M$ and that this type is axiomatized by $\{X_\alpha(y) : \alpha < \kappa\}$.
  \end{claimproof}

\vspace{1em}

\noindent \emph{Claim 3.}  $p$ is an heir over $M$.
\begin{claimproof}
  Fix an $M$-formula $\psi(y,z)$. Once again, there must have been a stage $\alpha$ at which $\psi^{\beta(\alpha)}_{\gamma(\alpha)}$ was defined and equal to $\psi$. Suppose that there is a $d$ in the monster such that $\psi(y,d) \in p(y)$. Since $\Uc$ is everywhere somewhere dense, this implies that the conditions for the first bullet point were met (with $N=M$ and $S=R$ and the parameter $d$). Therefore we added a parameter $e$ to $M_{\alpha+1}$ such that $\psi(b,e)$ holds for all $b \in X_{\alpha+1}(M_{\alpha+1})$. By elementarity, this implies that $\psi(b,e)$ holds for all $b \in X_{\alpha+1}(M)$. Therefore $\psi(y,e) \in p(y)$. Since we can do this for any $M$-formula $\psi(y,z)$, we have that $p(y)$ is an heir of $p\res M$ over $M$.
\end{claimproof}

\vspace{1em}

For any $n$, the set of finite right-combs of cardinality $n$ in $(b_\sigma)_{\sigma < \omega}$ is definable in the language $\Lc_0$. Therefore, by elementarity, we have that for any right-comb $Z \subseteq R$, $\{\varphi(x,b) : b \in Z\}$ is consistent. Likewise, being a $\tleq$-increasing sequence of size $n$ is a definable property, so for every path $Z$ in $R$, $\{\varphi(x,b) : b \in Z\}$ is $k$-inconsistent. 

\vspace{1em}

 \noindent \emph{Claim 4.} $\varphi(x,y)$ does not $p$-Kim-divide.
\begin{claimproof}
  The set $Y \coloneqq \bigcup_{\alpha < \omega}\{b \in R : b \tgeq c_\alpha \frown 0\}$ is in $\Uc$. Likewise, for each $\alpha < \kappa$, we have that the set $\{b \in R : b \tgeq c_\alpha\}$ is in $\Uc$. Note that if $\alpha_0,\alpha_1,\dots,\alpha_{n-1}$ is an increasing sequence of ordinals and if $b_i \tgeq c_{\alpha_i}\frown 0$ for each $i<n$, then $\{b_i:i<n\}$ is a right-comb.
  
  Let $(e_n)_{n<\omega}$ be a Morley sequence generated by $p$. By the argument in the paragraph just before Claim 4, we have that for any finite right-comb $Z \subseteq Y$, $\{\varphi(x,b) : b \in Z\}$ is consistent. Suppose that for some $n < \omega$ (possibly $0$), we've shown that for every finite right-comb $Z\subseteq Y$, $\{\varphi(x,b) : b \in Z\}\cup \{\varphi(x,e_i): i< n\}$ is consistent. For any such $Z$, there is an $\alpha$ such that for any $d \in Y$ with $d \tgeq c_\alpha$, $Z \cup \{d\}$ is a right-comb. This implies that $\{\varphi(x,b) : b \in Z\} \cup \{\varphi(x,e_i): i < n+1\}$ is consistent. Hence, by induction, we have that $\{\varphi(x,e_i) : i < \omega\}$ is consistent and so $\varphi(x,y)$ does not $p$-Kim-divide.
\end{claimproof}

\vspace{1em}

Finally, note that since $\{c_\alpha : \alpha < \kappa\} \in \Vc$ and since $\{\varphi(x,c_\alpha) : \alpha < \kappa\}$ is $k$-inconsistent, we have that $\{\varphi(x,e_i) : i < \omega\}$ is $k$-inconsistent for any Morley sequence $(e_i)_{i<\omega}$ generated by $q$. Therefore $p$ and $q$ satisfy the required conditions.

\vspace{1em}

\noindent \emph{Proof of 2.} Say that an $M_\beta$-formula $\psi$ is \emph{covered} by $\alpha < \kappa$ if $\psi^{\beta(\alpha)}_{\gamma(\alpha)}$ was defined and equal to $\psi$ at stage $\alpha$. Since $\kappa$ is regular, we have by a standard argument that there is a club $C \subseteq \kappa$ such that for every $\alpha \in C$, every $\beta < \alpha$, and every $M_\beta$-formula $\psi$, there is a $\gamma < \alpha$ such that $\psi$ is covered by $\gamma$.

Fix some $\mu \in C$. Let $A_\mu = \bigcup_{\alpha < \mu}M_\alpha$. We can now repeat the proofs of the claims in the proof of 1. In particular, we get from Claim 2 that $\{X_\alpha(y): y < \mu\}$ axiomatizes a complete type. Therefore we added an element $d_\mu$ to $M_\mu$ realizing this type. From Claims 3 and 4, we get that there is an $A_\mu$-heir-coheir $p_\mu(y)$ extending $\tp(d_\mu/A_\mu)$ such that $\varphi(x,y)$ does not $p$-Kim-divide.

Let $(\mu(\e))_{\e<\kappa}$ be an enumeration of $C$ in order. Let $N_\e = A_{\mu(\e)}$, $e_\e = d_{\mu(\e)}$, and $s_\e= p_{\mu(\e)}$ for each $\e < \kappa$. Since $(e_\e)_{\e < \kappa}$ is an $\tleq$-increasing sequence, we have that $\{\varphi(x,e_\e) : \e < \kappa\}$ is $k$-inconsistent, so we are done.
\end{proof}

For the following recall that on a measurable cardinal $\kappa$, we can find a normal ultrafilter $\Uc$ (i.e., for any sequence $(X_\alpha)_{\alpha < \kappa}$ of elements of $\Uc$, we have that the diagonal intersection $\Delta_{\alpha < \kappa}X_\alpha \coloneqq \{\beta < \kappa : \beta \in \bigcap_{\alpha < \beta}X_\alpha\}$ is an element of $\Uc$). In particular $\Uc$ is $\kappa$-complete and also has the property that every $X \in \Uc$ is stationary, implying that every club in $\kappa$ is in $\Uc$.

\begin{lem}\label{lem:large-cardinal-magic}
  Let $\kappa > |T|$ be a measurable cardinal. Let $M$ be a model of $T$ with $|M|= \kappa$. Let $(N_\alpha)_{\alpha < \kappa}$ be a continuous chain of elementary substructures of $M$ with $|N_\alpha| < \kappa$ for each $\alpha < \kappa$ and $\bigcup_{\alpha < \kappa} N_\alpha = M$. Let $(c_\alpha)_{\alpha < \kappa}$ be any sequence of elements of $M$.

  For any normal ultrafilter $\Uc$ on $\kappa$, there is an $X \in \Uc$ such that for every $\alpha \in X$,  $\tp(c_\alpha/ N_\alpha)$ is finitely satisfiable in $\{c_\beta : \beta < \alpha\}$ and $\{c_\beta : \beta < \alpha\} \subseteq N_\alpha$.
\end{lem}
\begin{proof}
  Let $p(y) = \{\psi(x,b) : \{\alpha < \kappa : \psi(c_\alpha,b)\} \in \Uc\} \in S_y(M)$. By a standard argument, there is a club $C \subseteq \kappa$ such that for any $\alpha \in C$, $\{c_\beta : \beta<\alpha\} \subseteq N_\alpha$. Note that since $C$ is a club, $C \in \Uc$.

  For each $\alpha$, let $p_\alpha = p\res N_\alpha$. By construction we have that for each $\psi(y) \in p_\alpha$, $\{\beta < \kappa : \psi(c_\beta)\} \in \Uc$. Since $|N_\alpha|<\kappa$ and since $\Uc$ is $\kappa$-complete, this implies that $\{\beta < \kappa : c_\beta \models p_\alpha\} \in \Uc$.

  Since $\Uc$ is normal, we have that $X\coloneqq C \cap\lim \kappa\cap\Delta_{\alpha < \kappa}\{\beta < \kappa:c_\beta \models p_\alpha\} \in \Uc$. Consider $\beta \in X$. We have that for every $\alpha < \beta$, $ c_\beta \models p_\alpha$ and $p_\alpha$ is finitely satisfiable in $\{c_\gamma : \gamma < \alpha\} \subseteq N_\gamma$. Since $\beta$ is a limit ordinal, this implies that $\tp(c_\beta/N_\beta)$ is finitely satisfiable in $\{c_\alpha:\alpha < \beta\} \subseteq N_\beta$. Since we can do this for any $\beta \in X$, we are done.
\end{proof}

\begin{defn}[{\cite[Def.~5.3]{Kaplan2019}}]
  A set $\Gamma(x)$ of formulas is a \emph{dual type} if there is some $k<\omega$ such that $\Gamma(x)$ is $k$-inconsistent. 
\end{defn}

\begin{prop}\label{prop:dual-char}
  Let $\kappa > |T|$ be a measurable cardinal. Let $M$ be the unique saturated model of $T$ of cardinality $\kappa$. Suppose that there is a dual type $\Gamma(x)$ over $M$ with $|\Gamma(x)| = \kappa$ and a club $C$ of small elementary substructures of $M$ such that for any $N \in C$, there is a $\varphi_N(x,c_N) \in \Gamma(x)$ with $c_N \notin N$ such that for some $N$-invariant $p_N(y) \supset \tp(c/N)$, $\varphi(x,c)$ does not $p$-Kim-divide.
  \begin{enumerate}
  \item\label{dual-CTP} If $p_N(y)$ is $N$-bi-invariant for each $N \in C$, then $T$ has CTP.
  \item\label{dual-ATP} If $p_N(y)$ is strongly $N$-bi-invariant for each $N \in C$, then $T$ has ATP.
  \end{enumerate}
\end{prop}
\begin{proof}
 The proofs of \ref{dual-CTP} and \ref{dual-ATP} are nearly identical. We will write the proof of \ref{dual-CTP}. To get the proof of \ref{dual-ATP}, just insert the word `strongly' in the appropriate places.

  By a standard argument we may assume that there is a continuous chain $(N_\alpha)_{\alpha < \kappa}$ of small elementary substructures of $M$ such that $M = \bigcup_{\alpha < \kappa}N_\alpha$ and $C = \{N_\alpha : \alpha < \kappa\}$. Let $\Uc$ be a normal ultrafilter on $\kappa$. For each $\alpha < \kappa$, let $\varphi_\alpha(x,c_\alpha) \in \Gamma(x)$ be a formula with $c_\alpha \notin N_\alpha$ such that for some $N_\alpha$-bi-invariant $p_\alpha(y) \supseteq \tp(c_\alpha/N_\alpha)$, $\varphi_\alpha(x,c_\alpha)$ does not $p_\alpha$-Kim-divide.

  By $\kappa$-completeness, there is a $Y \in \Uc$ and a formula $\varphi(x,y)$ such that for every $\alpha \in Y$, $\varphi_\alpha=\varphi$. By \cref{lem:large-cardinal-magic}, there is an $X \in \Uc$ such that for every $\alpha \in X$, $\tp(c_\alpha/N_\alpha)$ is finitely satisfiable in $\{c_\beta : \beta < \alpha\}$ and $\{c_\beta : \beta < \alpha\}\subseteq N_\alpha$.

  For any $\alpha \in X \cap Y$, we now have that for any global coheir $q$ finitely satisfiable in $\{c_\beta : \beta < \alpha\}\subseteq N_\alpha$, if $(e_i)_{i<\omega}$ is a Morley sequence generated by $q$, then $\{\varphi(x,e_i):i<\omega\}$ is $k$-inconsistent for some $k<\omega$. On the other hand, we have that $\varphi(x,e_0)$ does not $p_\alpha$-Kim-divide where $p_\alpha$ is an $N_\alpha$-bi-invariant type satisfying $p_\alpha \res N_\alpha = q\res N_\alpha$. Therefore $T$ has CTP.
\end{proof}

\begin{cor}\label{cor:CTP-char-dual}
  Fix a complete theory $T$. If there is a measurable cardinal $\kappa > |T|$, then the following are equivalent.
  \begin{enumerate}
  \item $T$ has CTP.
  \item There is a model $M \models T$ with $|M| = \kappa$, a dual type $\Gamma(x)$ over $M$, and a club $C$ of small elementary substructures of $M$ such that for any $N \in C$, there is a $\varphi(x,c) \in \Gamma(x)$ with $c \notin N$ such that for some $N$-bi-invariant $p(y) \supseteq \tp(c/N)$, $\varphi(x,c)$ does not $p$-Kim-divide.
  \item The same as 2, but with each $p(y)$ an $N$-heir-coheir.
  \end{enumerate}
\end{cor}

It seems likely that \cref{cor:CTP-char-dual} does not require a large cardinal, but nevertheless the question needs to be asked.

\begin{quest}
  Does \cref{cor:CTP-char-dual} hold without the existence of a large cardinal?
\end{quest}

Finally, we can essentially ask \cref{quest:ATP-Kim-char} again.

\begin{quest}
  Does \cref{cor:CTP-char-dual} hold for ATP if the types $p(y)$ are assumed to be strongly bi-invariant?
\end{quest}


\subsection{NATP implies generic stationary local character}
\label{sec:natp-implies-generic}

Simplicity, NTP$_2$, and NSOP$_2$ all have characterizations in terms of some kind of local character (which we will state in sub-optimal forms for the sake of exposition):
\begin{itemize}
\item $T$ is simple if and only if it satisfies \emph{local character}: There is a $\kappa$ such that for every global type $p(x)$, there is an $M$ with $|M|\leq \kappa$ such that $p(x)$ does not divide over $M$.
\item $T$ is NTP$_2$ if and only if it satisfies \emph{generic local character}: There is a $\kappa$ such that for every global type $p(x)$ and every $M$ with $|M| \leq \kappa$, there is an $N \succeq M$ with $|N| \leq \kappa$ such that for any $d$, if $d \indi_M N$, then $p \res Nd$ does not divide over $N$ \cite{Chernikov2014}.
\item $T$ is NSOP$_1$ if and only if it satisfies \emph{stationary local character}: There is a $\kappa$ such that for any global type $p(x)$, the set
  \[
    \{M \preceq \Ob : |M|\leq\kappa,~p(x)~\text{does not Kim-divide over}~M\}
  \]
   is stationary in the set of models of size less than $\kappa$ \cite{Kaplan2019}.
\end{itemize}
Note that in all three of these, (Kim-)dividing coincides with (Kim-)forking.

Given the existence of these characterizations, when Kruckman visited the logic group at the University of Maryland in 2020, we tried to come up with a reasonable mutual generalization of generic local character and stationary local character. The idea being that this would be a nice complementary approach to trying to mutually generalize NTP$_2$ and NSOP$_1$. %
We came to the following definition. %

\begin{defn}\label{defn:gslc}
  For any type $p(x)$ and any small $M,N \models T$ with $M \preceq N$, we write $\Xi(p,M,N)$ for the following condition:
  \begin{itemize}
  \item[$ $] For any $M$-formula $\varphi(x,y)$ and any $d$ such that $\varphi(x,d) \in p(x)$ and $d \indi_M N$, $\varphi(x,d)$ does not Kim-divide over $N$. 
  \end{itemize}
  We say that $T$ satisfies \emph{generic stationary local character}\footnote{Perhaps \emph{stationarily generic local character} would be a more correct name, but it doesn't quite roll off the tongue. It also doesn't generalize well to the club version discussed in \cref{quest:gclc}.} if for every $\lambda$, there is a $\kappa\geq \lambda$ such that for every $\kappa^+$-saturated model $O$, every type $p \in S(O)$, and every $M \preceq O$ with $|M|\leq \lambda$,
  \[
    \{N \preceq O: N \succeq M,~|N|\leq \kappa,~\Xi(p,M,N)\}
  \]
 is stationary in $[O]^\kappa \coloneqq \{X \subseteq O : |X| = \kappa\}$.
\end{defn}

As it turns out, however, generic stationary local character is probably too strong to characterize NBTP in that its failure actually implies ATP.

\begin{lem}\label{lem:bigger-bi}
  Fix $\kappa \geq |T|$. For any $\kappa^+$-saturated model $O$, any $M \preceq O$ with $|M| \leq \kappa$, and any $M$-invariant type $p(x)$, the set $\{N \subseteq [O]^\kappa:N\models T,~p~\text{is an }N\text{-heir}\}$ is a club in $[O]^\kappa$.
\end{lem}
\begin{proof}
  The set in question is clearly closed, so we just need to show that it is unbounded. Fix $M_0 \preceq O$ with $|M_0| = \kappa$. Given $M_n$, we can find a model $M_{n+1}$ with $|M_{n+1}| \leq \kappa$ such that for any $M_n$-formula $\varphi(x,y)$, if there is a $b$ in the monster such that $\varphi(x,b) \in p(\xbar)$, then there is a $c \in M_{n+1}$ such that $\varphi(\xbar ,c) \in p$. Let $N = \bigcup_{n<\omega}M_n$. We have that $p$ is an heir over $N$.
\end{proof}

\begin{prop}\label{prop:gen-stat-ATP}
  Suppose that $T$ does not satisfy generic stationary local character. Then there is a small model $N$, a formula $\varphi(x,d)$, a strongly $N$-bi-invariant type $p(y) \supset \tp(d/M)$, and an $N$-invariant $q(y)\supset \tp(d/M)$ such that $\varphi(x,d)$ $q$\nobreakdash-\hspace{0pt}Kim-divides but does not $p$-Kim-divide.
\end{prop}
\begin{proof}
  Let $\lambda$ witness that $T$ fails generic stationary local character. Let $\kappa = 2^{2^\lambda}$. Note that there are at most $\kappa$ $M$-invariant types for any $M$ with $|M|\leq \lambda$.

  There is a $\kappa^+$-saturated model $O$, a type $p(x) \in S(O)$, and an $M\preceq O$ with $|M|\leq \lambda$ such that $\{N \preceq O: N \succeq M,~|N|\leq \kappa,~\Xi(p,M,N)\}$ is not stationary in $[O]^\kappa$. This means that there is a club $C \subseteq [O]^\kappa$ such that for every model $N \in C$ with $N \succeq M$, there is a $\varphi(x,d) \in p(x)$ such that $d \indi_M N$ and $\varphi(x,d)$ Kim-divides over $N$. By induction, we can build a continuous chain $(N_i)_{i<\kappa^+}$ of elementary substructures of $O$ and a sequence $(\varphi_i(x,d_i))_{i<\kappa^+}$ of formulas in $p(x)$ such that for each $i<\kappa^+$, $d_i \indi_M N_i$ and $\varphi_i(x,d_i)$ Kim-divides over $N_i$. For each $i$, let $q_i(y) \supseteq \tp(d_i/N_i)$ be an $M$-invariant type.

  For each $j<\kappa^+$, let $D_j = \{i<\kappa^+ : i > j,~(\forall n < \omega)q_j^{\otimes n}~\text{is an}~N_i\text{-heir}\}$. By \cref{lem:bigger-bi}, each $D_j$ is a club in $\kappa^+$. Therefore the diagonal intersection $\Delta_{j<\kappa^+}D_j = \{i<\kappa^+: (\forall j < i)i \in D_j\}$ and also $E \coloneqq \lim\kappa^+\cap \Delta_{j<\kappa^+}D_j$ are clubs. For each $i \in E$, let $f(i)$ be the least $j$ such that $d_i \in N_j$. This is a regressive function, so by Fodor's lemma, there is a $k<\kappa^+$ and a stationary set $S \subseteq E$ such that $f(i) = k$ for all $i \in S$. Since the number of $M$-formulas is less than $\kappa$, there is a stationary set $S'\subseteq S$ such that for any $i$ and $j$ in $S'$, $\varphi_i = \varphi_j$. Since the number of $M$-invariant types is at most $\kappa$, there is a stationary set $S''\subseteq S'$ such that for any $i$ and $j$ in $S''$, $q_i=q_j$. Let $j = \min(S'')$ and let $\ell = \min(S'' \setminus \{j\})$.

  We now have that $(d_i)_{i\in S''\setminus\{j\}}$ is a Morley sequence generated by $q_\ell$ over $N_\ell$. Since $p(x)$ is consistent, we have that $\{\varphi_\ell(x,d_i) : i \in S''\setminus \{j\}\}$ is consistent as well. Furthermore, $q_j = q_\ell$ has the property that $q_j^{\otimes n}$ is an $N_j$-heir for each $n<\omega$. Therefore $q_j$ is strongly $N_\ell$-bi-invariant. Finally, we have that $\varphi_\ell(x,d_\ell)$ Kim-divides over $N_\ell$. Therefore $T$ has ATP by \cref{prop:Kim-failure-implies-CTP}.
\end{proof}

\begin{cor}\label{cor:if-NATP-then-gslc}
  If $T$ is NATP, then it satisfies generic stationary local character.
\end{cor}
\begin{proof}
  This follows immediately from Propositions~\ref{prop:Kim-failure-implies-CTP} and \ref{prop:gen-stat-ATP}.
\end{proof}

In our opinion the lesson of \cref{prop:gen-stat-ATP} and \cref{cor:if-NATP-then-gslc} is that it is unlikely that NCTP will be characterized by something like \cref{defn:gslc}. It is difficult to imagine how to tune the forbidden configuration so that it will build bi-invariant types but not also strongly bi-invariant types.\footnote{Naturally, this would be a non-issue if it does turn out that CTP and ATP are equivalent.}

Broadly speaking, Propositions~\ref{prop:easy-construction} and \ref{prop:countable-characterization-CTP} would seem to indicate that bi-invariance is expected generically at the level of formulas, whereas \cref{prop:gen-stat-ATP} indicates that strong bi-invariance is expected generically at the level of partial types.

Of course, \cref{cor:if-NATP-then-gslc} is unsatisfactory in a couple of different ways. First of all, it is not a characterization.

\begin{quest}
  If $T$ satisfies generic stationary local character, does it follow that $T$ is NATP?
\end{quest}

But also the local character characterizations of simplicity, NTP$_2$, and NSOP$_1$ have tighter cardinal bounds than we have stated. Our proof of \cref{prop:gen-stat-ATP} shows that if generic stationary local character fails at $\lambda$, then we can build an instance of ATP from any $(2^{2^\lambda})^+$\nobreakdash-\hspace{0pt}saturated model. For all three of the aforementioned conditions, it is known that this is witnessed by any $|T|^+$-saturated model. The proofs of these tighter bounds rely on more detailed structural understandings of the relevant class of tame theories, and in particular that the tameness condition is characterized by the associated notion of local character.

\begin{quest}
  If $T$ fails to have generic stationary local character, is this witnessed by models that are $|T|^+$-saturated?
\end{quest}

For NTP$_2$ in particular, the original statement of generic local character in \cite{Chernikov2014} did not require the big model to have any degree of saturation but instead required that each $\tp(d/N)$ extends to a strictly invariant type. We could make a similar statement here, more in the vein of \cref{prop:dual-char}, by dropping the requirement that the big model be saturated but require that each $\tp(d/N)$ extends to a strongly bi-invariant type. We did not opt to highlight this version of \cref{prop:gen-stat-ATP}, however, as the interesting thing about the proposition is the fact that we do not need to assume that some configuration of (strongly) bi-invariant types happens to exist, as we do in \cref{prop:dual-char}.

Finally, in the case of NSOP$_1$, we have done the relevant results in \cite{Kaplan2019} a bit of a disservice. In \cite{Kaplan2019}, Kaplan, Ramsey, and Shelah actually prove that in NSOP$_1$ theories, for any global type $p(x)$, the set of $|T|$-sized models $M \prec \Ob$ over which $p(x)$ does not Kim-divide is a club. This is an instance of the kind of dichotomous behavior you expect from a dividing line: Either every type is `good' on a club of small models or there is a type that is `bad' on a club of small models. Again, however, the proof of this relies heavily on a structural understanding of NSOP$_1$ theories, so it is entirely unclear if something like this would generalize to the present context, which leaves an obvious question.

\begin{quest}\label{quest:gclc}
  Say that a theory has \emph{generic club local character} if for any global type $p(x)$ and any $M \prec \Ob$ with $|M| \leq |T|$, the set $\{N \prec \Ob : N \succeq M,~|N|\leq|T|,~\Xi(p,M,N)\}$ is a club in $[\Ob]^{|T|}$.  Is generic stationary local character equivalent to generic club local character?
\end{quest}


\appendix
  \section{Dual results for NSOP$_1$}
\label{sec:Dual-NSOP1}

In \cite[Lem.~2.8]{Mutchnik-NSOP2}, Mutchnik shows that SOP$_2$ is equivalent to the following condition:
\begin{itemize}
\item[$ $] There is a $k<\omega$, a tree $(b_{\sigma})_{\sigma \in \omega^{<\omega}}$, and a formula $\varphi(x,y)$ such that for any path $\alpha \in \omega^\omega$, $\{\varphi(x,b_{\alpha \res n}) : n < \omega\}$ is consistent  but for any right-comb $C \subset \omega^{<\omega}$, $\{\varphi(x,b_\sigma) : \sigma \in C\}$ is  $k$-inconsistent.
\end{itemize}
Furthermore, by the main result of \cite{Mutchnik-NSOP2}, this is equivalent to SOP$_1$. In \cite{Some-Remarks-Kim-dividing-NATP}, Kim and Lee show that the above condition is equivalent after replacing the tree $\omega^{<\omega}$ with $2^{<\omega}$. The $2^{<\omega}$ form of the above condition is `dual' to \cref{defn:CTP} in the sense that it is identical after swapping the words `consistent' and `$k$-inconsistent.' Essentially all of the proofs given in this paper are insensitive to this duality, so we can freely conclude several new results for NSOP$_1$ theories. (Some of our results, such as \cref{prop:Kim-failure-implies-CTP}, are trivial when dualized, however.)

\begin{prop}[Dual of \cref{prop:countable-characterization-CTP} and \cref{prop:char-uncountable-languages} part \ref{char-uncountable-1}]
  If $T$ has SOP$_1$, then there is a model $M$, a formula $\varphi(x,y)$, an $M$-coheir $p(y)$, and an $M$\nobreakdash-\hspace{0pt}heir-coheir $q(y)$ such that $p\res M = q\res M$ and $\varphi(x,y)$ $q$-Kim-divides but does not $p$-Kim-divide.
\end{prop}

We also no longer need the measurable cardinal in the following result, as the hard part has already been done for us in \cite{Kaplan2019}.

\begin{prop}[Dual of \cref{cor:CTP-char-dual}]
  Fix a complete theory $T$. The following are equivalent.
  \begin{enumerate}
  \item\label{dual-dual-1} $T$ has SOP$_1$.
  \item\label{dual-dual-2} There is a model $M \models T$ with $|M| = \kappa$, a type $p(x)$ over $M$, and a club $C$ of small elementary substructures of $M$ such that for any $N \in C$, there is a $\varphi(x,c) \in p(x)$ with $c \notin N$ such that for some $N$-bi-invariant $p(y) \supseteq \tp(c/N)$, $\varphi(x,c)$ $p$-Kim-divides.
  \item\label{dual-dual-3} The same as 2, but with each $p(y)$ an $N$-heir-coheir.
  \end{enumerate}
\end{prop}
\begin{proof}
  The fact that \ref{dual-dual-1} implies \ref{dual-dual-3} follows from the dual of \cref{prop:char-uncountable-languages} part \ref{char-uncountable-2}. The fact that \ref{dual-dual-3} implies \ref{dual-dual-2} is obvious. Finally, the fact that \ref{dual-dual-2} implies \ref{dual-dual-1} follows from \cite[Thm.~1.1]{Kaplan2019}.
\end{proof}


\section{Characterization of coheirs in terms of invariant extensions in expansions}
\label{sec:coheir-inv-ext-exp}

Here we will give the proof that $\neg$\ref{dcl-sat} implies $\neg$\ref{invariant-in-expansions} in \cref{prop:coheir-characterization}. Recall the following: $B \inda_A C$ means that $\acl(AB)\cap \acl(AC) = \acl(C)$. $\inda$ satisfies all of the axioms of a strict independence relation except for possibly base monotonicity (see \cite[Sec.~1]{AdlerGeoIntro}\footnote{Although note that there are some errors in this source. See  \cite{conant2021separation} for a full account of the relevant results with correct proofs.}). In particular, this means that $\inda$ satisfies full existence: For any $A$, $B$, and $C$, there is a $C' \equiv_A C$ such that $B \inda_A C'$. %

First we will need a lemma.

\begin{lem}\label{lem:generic-split}
  Let $A$ be a set of parameters satisfying $A = \acl(A)$. Let $P$, $U$, and $U^\ast$ be fresh unary predicate symbols. Let $\Lc_0 = \Lc \cup \{P\}$ and let $\Lc_1 = \Lc_0 \cup \{U,U^\ast\}$. For any $\Lc_1$-formula $\varphi(\xbar)$, let $\varphi^\ast(\xbar)$ be the result of replacing each instance of $U$ with $U^\ast$ and each instance of $U^\ast$ with $U$ in $\varphi(\xbar)$.
  
  Let $\kappa = |A|+|\Lc|$. For any $\kappa^+$-saturated, $\kappa^+$-homogeneous\footnote{By $\kappa^+$-homogeneous we mean that for any tuples $\bbar$ and $\cbar$ realizing the same type with $|\bbar|=|\cbar|\leq \kappa$, there is an automorphism taking $\bbar$ to $\cbar$.} model $M \supseteq A$, there is an $\Lc_1$-structure $(N,P^N,U^N,U_\ast^N)$ such that
  \begin{itemize}
  \item $(N,P^N) \succeq_{\Lc_0} (M,A)$,
  \item $\{P^N,U^N,U_\ast^N\}$ forms a partition of $N$, and
  \item for any $\Lc_1(A)$-formula $\varphi(\xbar)$, $N \models (\forall \xbar \in P)(\varphi(\xbar) \toot \varphi^\ast(\xbar))$.
  \end{itemize}
\end{lem}
\begin{proof}
  Fix $A$ and $M$ as in the statement of the lemma. Note that since $M$ is $\kappa^+$\nobreakdash-\hspace{0pt}saturated as a model of $T$, we have that for any tuples $\bbar$ and $\cbar$ with $\tp_{\Lc}(\bbar/A) = \tp_{\Lc}(\cbar/A)$ and $|\bbar| = |\cbar|\leq \kappa$, $\bbar$ and $\cbar$ realize the same $\Lc\cup\{P\}$-type over $A$ as well (by a back-and-forth argument).

  Let $T_1$ be the theory consisting of the elementary diagram of $(M,P^M)$, an axiom asserting that $\{P,U,U^\ast \}$ is a partition of the universe, $\neg \exists x (P(x) \wedge \psi(x,e))$, and $(\forall \xbar \in P)(\varphi(\xbar) \toot \varphi^\ast(\xbar))$ for each $\Lc_1$-formula $\varphi$.

  Clearly we just need to show that $T_1$ is consistent. We will prove this by building an expansion of $M$ in a forcing extension that is a model of $T_1$. By absoluteness, this will imply that each finite subset of $T_1$ is consistent and so $T_1$ itself is consistent.

  Let $\Pb$ be a forcing poset whose conditions are pairs of the form $(B,C)$, where $B,C \subseteq M \setminus A$, $|B| \leq \kappa$, $|C|\leq \kappa$, and $B \cap C = \varnothing$. The ordering is given by extension, i.e., $(B',C') \leq (B,C)$ if and only if $B' \supseteq B$ and $C' \supseteq C$. Let $G$ be a generic filter for this poset and consider the forcing extension $V[G]$. Let $U^M = \bigcup\{B : (B,C) \in G\}$ and $U^M_\ast = \bigcup\{C: (B,C) \in G\}$. It is clear that $\{P^M,U^M,U_\ast^M\}$ forms a partition of $M$. Let $M_1$ be the $\Lc_1$-structure $(M,P^M,U^M,U_\ast^M)$. We just need to show that for each $\Lc_1$-sentence $\varphi$, $M_1\models \varphi \toot \varphi^\ast$. Fix an $\Lc_1$-formula $\varphi(\xbar)$  and a tuple $\abar \in A$ and suppose that $M_1 \models \varphi(\abar)$. By the truth lemma, there is a forcing condition $(B,C) \in G$ such that $(B,C) \Vdash \text{``}M_1 \models \varphi(\abar)\text{''}$. By full existence, we have that for any condition $(B',C') \leq (B,C)$, there is a $B''C''$ (in the monster in $V$) such that $BC \equiv_A B''C''$ and $B'C' \inda_A B''C''$ (i.e., $\acl(AB'C')\cap \acl(AB''C'') = \acl(A) = A$). In particular, $B'C' \cap B''C'' = \varnothing$ and so $B'C'' \cap C'B'' = \varnothing$. Since $M$ is $\kappa^+$\nobreakdash-\hspace{0pt}saturated in $V$, we may assume that $B''C'' \subseteq M$. Since the forcing poset is invariant under $\Aut(M/A)$, we have that $(C'',B'')\Vdash \text{``}M_1 \models \varphi^\ast(\abar)\text{''}$ and so $(B'C'',C'B'')$ forces this as well. Since we can do this for any $(B',C') \leq (B,C)$, we have that $(B,C)\Vdash \text{``}M_1 \models \varphi^\ast\text{''}$. Therefore $M_1\models \varphi(\abar) \toot \varphi^\ast(\abar)$. Since this is true for any $\abar \in P$, we have that $M_1 \models (\forall \xbar \in P)(\varphi(\xbar) \toot \varphi^\ast(\xbar))$. Since $\varphi(\xbar)$ was arbitrary, $M_1$ satisfies the required axiom schema.

    Finally, since $\psi(x,e)$ is not satisfied in $A = P^M$, we have that $M_1 \models \neg \exists x ( P(x) \wedge \psi(x,e))$.
\end{proof}

Now we can complete the proof of \cref{prop:coheir-characterization}.

\begin{proof}[Proof that $\neg$\ref{dcl-sat} implies $\neg$\ref{invariant-in-expansions} in \cref{prop:coheir-characterization}]
  Fix a set of parameters $A= \acl(A)$ and let $p(x)$ be an $A$-invariant type that is not finitely satisfiable in $A$. Fix an $\Lc$-formula $\psi(x,y)$ and a parameter $e$ such that $\psi(x,e) \in p(x)$ and $\psi(x,e)$ is not satisfiable in $A$. Let $\kappa = (|A|+|\Lc|)^+$. Let $M \supseteq A$ be a $\kappa^+$-saturated, $\kappa^+$-homogeneous model of $T$. By $\kappa^+$-saturation we may assume that the parameter $e$ is in $M$. Let $P$ be a new unary predicate. Expand $M$ with $P$ so that $P^M = A$.

  Let $N$ be the extension of $M$ guaranteed by \cref{lem:generic-split}. By a standard model-theoretic argument, we may assume (by passing to a sufficiently saturated and homogeneous elementary extension if necessary) that there is an automorphism $f \in \Aut(N/A)$ (of $N$ as an $\Lc$-structure) such that $f[U^N]= U_\ast^N$. (Note that $f$ does not necessary fix all of $P^N$.) Now let $E$ be the equivalence relation on $N$ with equivalence classes $\{P^N,U^N,U_\ast^N\}$. Consider the expansion $(N,E)$. Note that $f$ is still an automorphism of this structure. Note also that $(N,E)\models \neg \exists x ( x \mathrel{E} a \wedge \psi(x,e))$ for any $a \in A$. Let $q(x)$ be an $\Aut((N,E)/A)$-invariant extension of $p\res N$. Since $\psi(x,e) \in q(x)$, we must have that $q(x)$ concentrates on one of the $E$-equivalence classes other than the class containing $A$, but these classes are not fixed by $f$, so we have a contradiction. Therefore no such extension can exist.
\end{proof}

One thing to note about this proof is that in the expansion, the set $A$ is not algebraically closed in $T_1^{\mathrm{eq}}$. This raises a question.

\begin{quest}
  Which invariant types $p(x)$ satisfy the following property?
  \begin{itemize}
  \item[$ $] For any model $M \supseteq A$ and any expansion $M^\dagger$ of $M$, $p\res M$ has a completion in $S(M^\dagger)$ that is $\Autf(M^\dagger/A)$-invariant.
  \end{itemize}
\end{quest}

\section{An example of Kim-dividing only along non-extendibly invariant types}
\label{sec:Kim-dividing-extendibly}

Let $\Lc$ be a language with three sorts: $G$ (for graph), $O$ (for orders), and $P$ (for points). $G$ has a unary relation $U$ and a binary relation $R$. We have four unary functions, which we will think of as two unary functions with codomain $G^2$: $\langle f_O,g_O \rangle : O \to G^2$ and $\langle f_P,g_P \rangle : P \to G^2$. Given $x$ and $y$ in $G$, we'll write $O_{x,y}$ for the $\langle f_O,g_O \rangle$-preimage of $\langle x,y \rangle$ and we'll write $P_{x,y}$ for the $\langle f_P,g_P \rangle$-preimage of $\langle x,y \rangle$. Finally, we have a ternary relation $<$ on $O \times P^2$, which we'll write as a parameterized binary relation $y <_x z$ for $x \in O$.

Let $Q(x,y,z,w)$ be the formula that says $|\{x,y,z,w\}|=4$, $\{x,y,z,w\} \subseteq U$, and there is an $R$-edge from some element of $\{x,y\}$ to some element of $\{z,w\}$.

Let $T_0$ be the following universal theory: 
\begin{itemize}
\item $R$ is a triangle-free graph relation, 
\item $U(f_O(x))$, $U(g_O(x))$, $U(f_P(x))$, and $U(g_P(x))$ always hold,
\item if $Q(x,y,z,w)$, then for any $\ell \in O_{x,y}$, $<_{\ell}$ is a linear order on $P_{z,w}$, and
\item for any $\ell \in O_{x,y}$, if $u <_\ell v$ for some $u$ and $v$, then $\langle f_P(u),g_P(u) \rangle = \langle f_P(v),g_P(v) \rangle$.
\end{itemize}
It is not hard but also not entirely pleasant to establish that the finite models of $T_0$ form a \Fraisse\ class with free amalgamation. Let $T$ be the theory of its \Fraisse\ limit. Note in particular that $T$ has quantifier elimination.


Fix a model $M$ of $T$. Fix $m \in \neg U(M)$ and $b$ and $c$ outside of $M$ such that $\neg (b \mathrel{R} c)$, the only edge between $M$ and $\{b,c\}$ is $c \mathrel{R} m$, and $U(b)$ and $U(c)$ hold.

\begin{lem}
  The formula $x \mathrel{R}b \wedge x \mathrel{R}c$ Kim-divides over $M$.
\end{lem}
\begin{proof}
  Let $p(y,z)$ be the global $M$-invariant type extending $\tp(bc/M)$ entailing for all $d$ in the monster,
   $y \mathrel{R}d$ if and only if $d \notin M$ and $d \mathrel{R} m$, and
   $z \mathrel{R} d$ if and only if $d = m$.
  By quantifier elimination, this entails a complete $M$-invariant.

  If we find $b'c' \models p \res M bc$, then we'll have $b' \mathrel{R}  c$, implying that $x \mathrel{R} b\wedge x \mathrel{R} c \wedge x \mathrel{R}b' \wedge x  \mathrel{R}c'$ is inconsistent.
\end{proof}

\begin{lem}\label{lem:lem-2}
  If $x \mathrel{R}b \wedge x \mathrel{R}c$ Kim-divides with respect to an $M$-invariant type $p(y,z) \supseteq \tp(bc/M)$, then $p(y,z) \vdash Q(y,z,b,c)$.
\end{lem}
\begin{proof}
  For any $M$-invariant type $p(y,z) \supseteq \tp(bc/M)$, $p(y,z) \vdash |\{y,z,b,c\}| = 4$. Furthermore, we necessarily have that $p(y,z) \vdash U(y)\wedge U(z)$, so the only thing to check is that if $b'c' \models p \res M bc$, then there is an $R$-edge between some element of $\{b,c\}$ and some element of $\{b',c'\}$.

  Suppose that this doesn't happen. Then if $(b_ic_i)_{i<\omega}$ is a Morley sequence generated by $p(y,z)$, we'll have that there are no $R$-edges between any pair of elements of $\{b_i,c_i:i<\omega\}$, implying that $\{x \mathrel{R} b_i\wedge x \mathrel{R} c_i : i< \omega\}$ is consistent, contradicting the fact that this formula Kim-divides along Morley sequences generated by $p(y,z)$. Therefore $p(y,z) \vdash Q(y,z,b,c)$.
\end{proof}

Let $\Ob$ be the monster model of $T$.

\begin{lem}
  $P_{b,c}(\Ob)$ is an $M$-indiscernible set.
\end{lem}
\begin{proof}
  First note that $Q(y,z,b,c)$ does not hold for any $\{y,z\} \subseteq M$. Therefore for any $\ell \in O(M)$, $<_\ell$ is trivial on $P_{b,c}(\Ob)$. By quantifier elimination, this implies that any two $n$-tuples of distinct elements of $P_{b,c}(\Ob)$ have the same type over $M$, so $P_{b,c}(\Ob)$ is an $M$-indiscernible set.
\end{proof}

Let $d$ be an element of $O_{b,c}(\Ob)$.

\begin{lem}
  If $p(y,z,w) \supseteq \tp(bcd/M)$ is an $M$-invariant type, then $p(y,z,w) \vdash \neg Q(y,z,b,c)$.
\end{lem}
\begin{proof}
  Assume for the sake of contradiction that $p(y,z,w)\vdash Q(y,z,b,c)$. Define a relation $<^p$ on $P(\Ob)$ by $e <^p f$ if and only if $p(y,z,w) \vdash e <_w f$. Note that $<^p$ is clearly $M$-invariant. Since $p(y,z,w)\vdash Q(y,z,b,c)$, the restriction of $<^p$ to $P_{b,c}(\Ob)$ needs to be a linear order, but there are no $M$-invariant linear orders on $P_{b,c}(\Ob)$, since it is an $M$-indiscernible set.
\end{proof}

\begin{prop}\label{prop:counterexample-Kim-dividing}
  If $x \mathrel{R} b \wedge x \mathrel{R}c$ Kim-divides along an $M$-invariant type $p(y,z) \supseteq \tp(bc/M)$, then $p(y,z)\cup \tp(bcd/M)$ has no $M$-invariant completions.
\end{prop}

One thing to note about this example is that it does in fact have ATP (since it interprets the Henson graph). It follows from \cite[Thm.~3.10]{Some-Remarks-Kim-dividing-NATP} that \cref{prop:counterexample-Kim-dividing} cannot occur in an NATP theory.




\bibliographystyle{plain}
\bibliography{../ref}

\end{document}